\documentclass[twoside,11pt]{amsart}

\usepackage[cmtip,all]{xy}
\usepackage{amssymb, xspace, enumerate, times, color, hyperref}

\usepackage{graphicx}
\usepackage{algorithm}
\usepackage{algpseudocode}
\input xypic
\input xy
\xyoption{all}
\def\sqr#1#2{{\vcenter{\hrule height.#2pt
        \hbox{\vrule width.#2pt height#1pt \kern#1pt
                \vrule width.#2pt}
        \hrule height.#2pt}}}

\newtheorem{Theorem}{Theorem}[section]
\newtheorem{Lemma}[Theorem]{Lemma}
\newtheorem{Corollary}[Theorem]{Corollary}
\newtheorem{Proposition}[Theorem]{Proposition}

\newtheorem{Notation}[Theorem]{Notation}
\newtheorem{Assumptions and Discussion}[Theorem]{Assumptions and Discussion}

\newtheorem{Remark}[Theorem]{Remark}
\newtheorem{Example}[Theorem]{Example}
\newtheorem{Definition}[Theorem]{Definition}

\def\m{{\mathfrak m}}

\def\q{{\mathfrak q}}
\def\p{{\mathfrak p}}

\newcommand{\wdt}[1]{\widetilde{#1}}
\newcommand{\wdh}[1]{\widehat{#1}}
\newcommand{\symp}[1]{#1^{(\ell)}}

\newcommand{\dual}{^*}
\newcommand{\trunc}[1]{_{[#1]}}
\newcommand{\elong}[1]{^{#1}}
\newcommand{\supp}[1]{{\rm supp}(#1)}

\newcommand{\sfp}{SF}

\def\ZZ{{\mathbb Z}}
\def\NN{{\mathbb N}}

\newcommand{\R}{\mathcal{R}}
\newcommand{\B}{\mathcal{B}}
\newcommand{\F}{\mathcal{F}}

\def\Llra{\Longleftrightarrow}

\def\Lra{\Longrightarrow}
\def\lra{\longrightarrow}

\newcommand{\be}{\begin{equation*}}
\newcommand{\ee}{\end{equation*}}
\newcommand{\bee}{\begin{equation}}
\newcommand{\eee}{\end{equation}}

\def\h{{\rm ht}}

\def\Ass{{\rm Ass}}

\def\sdef{{\rm sdefect}}

\def\LCM{{\rm LCM}}
\def\GCD{{\rm GCD}}

\def\sub{\subseteq}

\newcommand{\Cmatroid}{{$C$-matroidal }}
\newcommand{\Cmatroidal}{{$C$-matroidal }}
\newcommand{\Bmatroid}{{$B$-matroidal }}
\newcommand{\unif}{{uniformity--}}
\def\ut{{\rm ut}}

\begin{document}

\title{The structure of Symbolic Powers of Matroids} 

\author{Paolo Mantero}
\address{University of Arkansas, Fayetteville, AR 72701}
\email{pmantero@uark.edu}
\thanks{The first author was partially supported by Simons Foundation Grant \#962192.}

\author{Vinh Nguyen}
\address{University of Arkansas, Fayetteville, AR 72701}
\email{vinhn@uark.edu}
\thanks{}

\begin{abstract} 
	We describe the structure of the symbolic powers $\symp{I}$ of the Stanley--Reisner ideals, and cover ideals, $I$, of matroids.  
	We (a) prove a structure theorem describing a minimal generating set for every $\symp{I}$; (b) describe the (non--standard graded) symbolic Rees algebra $\R_s(I)$ of $I$ and show its minimal algebra generators have degree at most $\h\,I$; (c) provide an explicit, simple formula to compute the largest degree of a minimal algebra generator of $\R_s(I)$; (d) provide algebraic applications, including formulas for the symbolic defects of $I$, the initial degree of $\symp{I}$, and the Waldschmidt constant of $I$;   
	(e) provide a new algorithm allowing fast computations of very large symbolic powers of $I$.
	
	One of the by--products is a new characterization of matroids in terms of minimal generators of $\symp{I}$ for some $\ell\geq 2$. In particular, it yields a new, simple characterization of matroids in terms of the minimal generators of $I^{(2)}$. This is the first characterization of matroids in terms of $I^{(2)}$, and it complements a celebrated theorem by Minh--Trung, Varbaro, and Terai--Trung which requires the investigation of  homological properties of $\symp{I}$ for some $\ell\geq 3$.

\end{abstract}

\maketitle
\section{Introduction}

This is the first of a series of paper aimed at unveiling the structure of the symbolic powers of Stanley--Reisner ideals, or cover ideals, of matroids. Matroids are ubiquitous in the mathematics literature, partly because they capture the notion of ``independence" in a variety of different contexts, thus providing a uniform theory which can then be applied to different areas of mathematics. In this manuscript we specifically investigate one of the relations between matroids and Commutative Algebra.
We study {\em \Cmatroid ideals}, i.e. ideals which are either the Stanley--Reisner ideal or the cover ideal of a matroid, see Definition \ref{Def-Matroidal}. In this introduction, for the sake of simplicity, we just state our results for Stanley--Reisner ideals instead of \Cmatroid ideals.

There is a vast literature investigating the Stanley--Reisner ideals $I_\Delta$ associated to a simplicial complex $\Delta$.  By identifying the independence complex of a matroid with the matroid itself, one can consider matroids as a subcollection of all simplicial complexes. Considering the many known characterizations of matroids, researchers investigated the question of characterizing matroids in terms of Stanley--Reisner ideals. A very elegant answer was provided around 2011 in terms of symbolic powers of $I_\Delta$. 

We briefly recall here that symbolic powers arise naturally in several contexts in Commutative Algebra and Algebraic Geometry, in particular in the study of multi--variate polynomial interpolation problems. While they have been widely investigated over the past 70 years, in general symbolic powers are very challenging to describe -- even for squarefree monomial ideals! --  and are, in fact, the main characters in a number of celebrated open questions in the literature. To name a few: Nagata's conjecture (raised in connection to Nagata's counterexample to Hilbert's 14th problem) \cite{Nag} \cite{CHMR}, Eisenbud--Mazur conjecture (motivated by the investigation of the (non--)existence of non--trivial evolutions in the theory of Galois representations) \cite{EM} \cite[Conj~2.25]{DDGH}, Chudnovsky's conjecture (related to polynomial interpolation problems) \cite{Chu} \cite[Conj~1.1]{BGHN} \cite[Conj~1.3]{FMX}, a number of conjectures by Harbourne and Huneke \cite{HaHu}, and Conforti--Cornuejols conjecture (originally stated in 1993 in the context of combinatorial optimization theory and later translated in Commutative Algebra language by Gitler--Valencia--Villarreal \cite{GVV}) \cite[Conj~1.6]{Cor} \cite[Section~4.2]{DDGH} -- which even had a \$5,000 prize attached if it were solved, or disproved, by December 2020 \cite{Cor}.
\medskip

Back to matroids, the following celebrated theorem, first proved independently by Minh--Trung and Varbaro \cite{MT} \cite{Var}, and later strengthened by Terai--Trung \cite{TT}, provides a characterization of when a simplicial complex $\Delta$ is a matroid in terms of (good homological properties of) some symbolic power of the Stanley--Reisner ideal $I_\Delta$. See \cite{LM} for a new, elementary proof of the following theorem.

\begin{Theorem}[{\cite[Theorem~3.6]{TT}}]\label{symbintro}
	Let $\Delta$ be a simplicial complex on $[n]$, let $I_{\Delta}\subseteq R:=k[x_1,\ldots,x_n]$ be its Stanley--Reisner ideal. The following are equivalent:
	\begin{enumerate}
		\item[$(1)$] $\Delta$ is a matroid,		
		\item[$(2)$] $R/I_{\Delta}^{(\ell)}$ is Cohen-Macaulay for every $\ell \ge 1$,
		\item[$(3)$] $R/I_{\Delta}^{(\ell)}$ satisfies Serre's property $(S_2)$ for every $\ell \ge 1$,
		\item[$(4)$] $R/I_{\Delta}^{(\ell)}$ satisfies  Serre's property $(S_2)$ for some $\ell \ge 3$.
	\end{enumerate}
\end{Theorem}

As stated in the abstract, our main objective is to develop a theory allowing one to virtually answer any question about {\em any} symbolic power $\symp{I_{\Delta}}$ of the Stanley--Reisner ideal $I_\Delta$ of {\em any} matroid. In this paper we make a first fundamental step in this direction by providing a structure theorem that completely describes a minimal generating set of $\symp{I_{\Delta}}$ for any $\Delta$ and $\ell$. In an upcoming paper by the same authors, this structure theorem is a cornerstone to answer several other questions about $\symp{I_{\Delta}}$ \cite{MN2}.

In any polynomial ring, besides the complete intersection ideals (where ordinary and symbolic powers agree), there are very few classes of ideals $I$ for which a description of the minimal generators of all symbolic powers $\symp{I}$ are known. Ideals defining star configurations -- see \cite[Def~2.1]{GHMstar} or the line after Definition \ref{Def-Matroidalconfig} -- are one such class, see \cite[Thm~3.6]{GHMstar} for the case of $\h\,I=2$ and \cite[Thm~4.9]{M} for the general case. 

Since ideals of star configurations are Stanley--Reisner ideals of uniform matroids on some groundset $[n]$, our first main result generalizes the structure theorem \cite[Thm~4.9]{M} to any matroid. The generalization is non--trivial for several reasons, including that uniform matroids (and, thus, ideals of star configurations) have a natural action of the symmetric group $S_n$ on them, a fact widely exploited in \cite{BD+} and \cite{M}, while, conjecturally, asymptotically almost every matroid on $[n]$ is {\em asymmetric}, i.e. it is not fixed by any non--trivial subgroup of $S_n$, e.g. \cite[Conj~1.2]{LOSW}.

It is our understanding that our first main result provides the largest class of ideals for which a minimal generating set of the symbolic powers are explicitly described.  
Recall that for a monomial ideal $L$, $G(L)$ denotes the unique minimal generating set of $L$ consisting of monomials.

\begin{Theorem}(Structure Theorem~\ref{MatroidSymPowerThm})\label{SymPowerintro}
	Let $\Delta$ be a matroid, and let $I = I_{\Delta}$. Then $G(\symp{I})$ consists precisely of the monomials $m$ of the form $m = m_1 \cdots m_s$, where each $m_i$ is a squarefree minimal generator of $I^{(c_i)}$ for some $1 \leq c_i \leq \h \, I$ with $\sum c_i = \ell$, and  $\supp{m_1} \supseteq ... \supseteq \supp{m_s}$.\end{Theorem}
In fact, the theorem provides a ``tower" structure for the minimal generators of $\symp{I}$, resembling the ones seen in standard monomial theory. It states that all minimal generators are obtained by stacking squarefree generators of smaller symbolic powers with nested support.

In the second main result of this paper we prove that the presence of this ``tower" structure for some symbolic power $\symp{I_\Delta}$ with $\ell\geq 2$ actually characterizes when $\Delta$ is a matroid. This may be useful as the ``tower" structure is particularly simple to describe or analyze for $I^{(2)}$.

\begin{Theorem}\label{G2intro} Let $\Delta$ be a simplicial complex. Let $I$ be the Stanley Reisner $I_{\Delta}$. The following are equivalent: \begin{enumerate}
		\item $\Delta$ is a matroid, 
		\item For all $\ell \geq 1$, the elements of $G(\symp{I})$ have the form described in Theorem \ref{SymPowerintro},
		\item For some $\ell \geq 2$, the elements of $G(\symp{I})$ have the form described in Theorem \ref{SymPowerintro},
		\item $G(I^{(2)}) = \{m_1^2,\ldots,m_r^2\}\cup G(\sfp_2(I))$, where $G(I)=\{m_1,\ldots,m_r\}$ and $\sfp_2(I)$ is the ideal of squarefree monomials in $I^{(2)}$.
	\end{enumerate}\end{Theorem}

Similarly to how the Rees algebra of $I\subseteq R$ provides a ring which can be used to study all ordinary powers of $I$ at once, the symbolic Rees algebra $\R_s(I)=\bigoplus_{\ell\geq 0}I^{(\ell)}t^{\ell} \subseteq R[t]$ allows one to study all symbolic powers at once. For most ideals $I$, the ring $\R_s(I)$ is not Noetherian. However, Herzog, Hibi, and Trung proved in 2007 that $\R_s(I)$ is Noetherian for any monomial ideal $I$ \cite{HHT}. It follows that $\R_s(I)$ has finitely many minimal generators as an algebra over $R$. We call the largest $t$-degree of a minimal $R$-algebra generator of $\R_s(I)$ the {\em symbolic Noether number of $I$} as it is the symbolic Rees algebra analogue of the Noether number in Invariant Theory. The symbolic Noether number of $I$ is then a measure of the complexity of the symbolic powers of an ideal. Herzog, Hibi, and Trung proved that even for squarefree monomial ideals $I$ the symbolic Noether number of $I$ can be extremely large -- in fact, it cannot be bounded by any linear function in the number of variables of the ring \cite[Ex.~5.5]{HHT}. 

In stark contrast with the above, the description given in Theorem \ref{SymPowerintro} implies that the symbolic Noether number of $\Delta$ is actually small, in fact, it is bounded above by the corank of $\Delta$, which is the height of $I_\Delta$, when $\Delta$ is a matroid. 
We then refine this result by completely characterizing the symbolic Noether number of $I_\Delta$, see Theorem \ref{NoetherNumber}.

Then, to illustrate a potential use of the above theorems, we provide several different applications.  
The {\em symbolic defects} of an ideal $I$ have been introduced in \cite{GGSV} to measure how far is $\symp{I}$ from $I^\ell$. These integers are extremely challenging to compute, e.g. they are not even known for squarefree monomials of height 2 \cite{DG}. One of the very few cases where they are known is for ideals of star configurations \cite[Cor~4.13]{M}. Here, we provide a formula in terms of the minimal number of generators of $\symp{I}$, Theorem \ref{Thm-Sdef}. We introduce the notion of ideals of {\em maximal symbolic defects}, see Definition \ref{Def-MaxSymDef}, which may be of independent interest, and, in another main result, Theorem \ref{Thm-Sdef2}, we characterize the (many) matroids $\Delta$ for which some (equivalently, every) symbolic defect of $I_\Delta$ is maximal. 

We also provide combinatorial descriptions for the initial degrees of $\symp{I_\Delta}$ and the Waldschmidt constant of $I_\Delta$, see Corollary \ref{Waldschmidt-SqFree}, and a \texttt{Macaulay2} algorithm to compute symbolic powers of $\symp{I}$. The algorithm complements existing algorithms, as it is very fast for large $\ell$, see Section \ref{Section-Alg}. 

Of independent interest may be the notion of {\em \unif threshold} of a pure simplicial complex. 
The \unif threshold of a matroid $\Delta$ can be deduced from the girth of $\Delta$. We employ the \unif threshold to characterize paving matroids and refine some of the above formuals; this yields, for instance, simpler formulas and more accurate bounds for ideals associated to paving or sparse paving matroids. Our results apply to ground fields $k$ of any characteristic.

\medskip

The structure of the paper is the following. In Section 2 we recall basic definitions and establish most of the notation used in the paper. In Section 3 we prove Theorems \ref{SymPowerintro} and \ref{G2intro}. In Section 4 we study the structure of the symbolic Rees algebra of ideals $I$ associated to matroids and the symbolic Noether number. In Section 5 we study the symbolic defects of $I$. In Section 6 we study other invariants associated to $\symp{I}$. In Section 7 we introduce the \unif threshold, and provide simpler formulas in the case of paving and sparse paving matroids. In Section 8 we provide a fast \texttt{Macaulay2} algorithm to compute $\symp{I}$.

{\bf Acknowledgment.} When this paper was in preparation, we had a private communication with M. DiPasquale, L. Fouli, A. Kumar, and S. Toh\v{a}neanu where we informed them about some of the results we proved in this paper and in our upcoming paper \cite{MN2}. Later, they have shared with us a preprint of their paper \cite{DFKT}, where they have drawn  some interesting connections between $\R_s(I_\Delta)$ and coding theory, when $\Delta$ is a matroid, and whose results are obtained independently from ours.

\section{Preliminaries and background}

In this section we collect basic definitions and facts, and establish some notation. Following standard notation in matroid theory, for a set $F$ and an element $x$, we will write $F-x$ for $F-\{x\}$ and $F\cup x$ for $F\cup \{x\}$. 

\subsection{Matroids}
One of the most common ways to define a matroid is by specifying its bases.  We list out three equivalent ways to do so. 

\begin{Definition}\label{Def-Matroid-Basis} A {\em matroid} $M$ on a ground set $V$ consists of a non-empty collection $\B$ of subsets of $V$ whose elements satisfy any of the following equivalent basis exchange properties.
	
	\begin{enumerate} 
		\item (Basis exchange property) 
		For any $F,G \in \B$ and for any $v \in F - G$, there exist a $w \in G - F$ such that $(F - v) \cup w \in \B$.
		\item (Symmetric basis exchange property) 
		For any $F,G \in \B$ and for any $v \in F - G$, there exist a $w \in G - F$ such that both $(F - v) \cup w \in \B$ and $(G-w) \cup v \in \B$.
		\item (Symmetric multi-basis exchange property) 
		For any $F,G \in \B$, and for any subset $A \sub F - G$, there exist a subset $B \sub G - F$ such that both $(F - A) \cup B \in \B$ and $(G - B) \cup A \in \B$.
	\end{enumerate}
	The elements of $\B$ are called {\em bases} of $M$. Any subset of any basis is called an {\em independent set} of $M$. The {\em rank} of any subset $A\sub V$, denoted $r(A)$, is the size of the largest independent set contained in $A$. One can show that all bases of $M$  have the same size, hence the rank of $V$ is equal to the size of any basis of $M$, and one defines this number $r(M)$ to be the {\em rank} of $M$. 
\end{Definition} 

In this paper, the set $V$ will be finite, so we will take $V$ to be $[n]:=\{1,\ldots,n\}$. Next we explain why, by slight abuse of notation, we can think of any matroid as a simplicial complex. 

\begin{Remark}
	To any matroid $M$ on $[n]$ one can associate a simplicial complex $\Delta_M$ on $[n]$, called the {\em independence complex} of the matroid, whose facets  are precisely the bases of $M$.  A simplicial complex $\Delta$ is {\em matroidal} if $\Delta=\Delta_M$ for some matroid $M$.

	It is easy to check that the above association is 1--to--1, i.e. if $\Delta_M=\Delta_N$, then $M=N$. Since in this paper we discuss ideals usually associated to simplicial complexes, then we will identify a matroid with its independence complex. 
	When we will write ``Let $\Delta$ be a matroid", we will mean ``Let $\Delta$ be a matroidal simplicial complex". 
	
\end{Remark}
We now gather other relevant definitions and properties of matroids. We will refer to the standard reference books of \cite{Oxley} or \cite{Welsh} for other well known properties of matroids.

\begin{Definition}\label{Def-Flats-Hyperplanes} Let $\Delta$ be a matroid on $[n]$ with rank function $r$. The {\em closuse} of a set $A\subseteq [n]$ is $\textrm{cl}(A) = \{ i\in [n] : r(A \cup i) = r(A)\}$. A  {\em flat} of $\Delta$ is a closed subset of $[n]$, i.e. a subset $A\subseteq [n]$ with $A={\rm cl}(A)$. A {\em hyperplane} of $\Delta$ is a flat of rank $r(M) - 1$. 
\end{Definition}

There is another definition of matroid based on its circuits. This definition will be the most useful for our purposes as there is a direct connection between the circuits of a matroid and generators of its associated Stanley--Reisner ideal.

\begin{Definition} Let $\Delta$ be a matroid on $[n]$. A subset $C \sub A$ is a circuit of $\Delta$ if it is a dependent set of $M$ that is minimal with respect to inclusion.
	
	Conversely, a collection $\mathcal{C}$ of subsets of $[n]$ is a set of circuits of a matroid $M$ on $[n]$ if and only if  $\mathcal{C}$ satisfies the following properties.
	\begin{enumerate}
		\item $\emptyset \notin \mathcal{C}$.
		\item For any pair $C_1,C_2 \in \mathcal{C}$, if  $C_1 \subseteq C_2$ then $C_1 = C_2$.
		\item For distinct circuits $C_1,C_2$ if $x\in C_1 \cap C_2,$ then $\exists C_3 \in \mathcal{C}$ such that $C_3 \subseteq (C_1 \cup C_2) - x$.
	\end{enumerate}
\end{Definition}

Next we define the dual of a matroid $\Delta$, loops, and coloops.
\begin{Definition}\label{Def-Dual-Matroid} Let $\Delta$ be a matroid on $[n]$. \begin{enumerate}
		\item The {\em dual} of $\Delta$, denoted $\Delta\dual$, is a matroid on $[n]$ whose collection of basis $\B\dual$ is defined as $$\B\dual = \{ [n] - F : F \in B\}.$$ 
		The rank function $r^*$ of $\Delta^*$ is called the {\em corank} function.
		\item A {\em loop} of $\Delta$ is an element $v \in [n]$ not contained in any basis of $\Delta$, i.e. $\{v\}$ is a circuit. 
		\item A {\em coloop} of $\Delta$ is an element $v \in [n]$ contained in every basis of $\Delta$, i.e. $v$ is a loop of $\Delta\dual$.
	\end{enumerate}
\end{Definition}

Next we discuss the truncations and elongations of a matroid.
\begin{Definition}
	Let $\Delta$ be a matroid on $[n]$ with rank function $r$. For any $h\in \NN_0$, 
	\begin{itemize}
		\item the {\em $h$-truncation $\Delta\trunc{h}$ of} $\Delta$ is the matroid whose bases are the independent sets of $\Delta$ of size $r(\Delta)-h$.
		\item the {\em elongation} $\Delta\elong{h}$ of $\Delta$ to rank (or to height) $r(\Delta)+h$ is the matroid whose independent sets are the subsets $H\subseteq [n]$ with $|H|-r(H)\leq h$.
	\end{itemize}
\end{Definition}

These two concepts are dual in the sense that the dual of the elongation of $\Delta$ to rank $r(\Delta) +h$ is the $h$-th truncation of the dual of $\Delta$, i.e. $(\Delta^{h})\dual=(\Delta\dual)\trunc{h}$.

\subsection{Stanley--Reisner and Cover Ideal of Matroids}

Recall that for this paper $k$ is any field, $[n]=\{1,\ldots,n\}$, $R=k[x_1,\ldots,x_n]$, and $\m=(x_1,\ldots,x_n)$.  
For a simplicial complex $\Delta$ on $[n]$ we write 
\begin{itemize}
	\item $\mathcal{F}(\Delta)$ for the set of all facets of $\Delta$,
	\item $\p_F=(x_i\,\mid\, i\in F)$ for any $F\in \Delta$. 
\end{itemize}

We now introduce two squarefree monomial ideals commonly associated to a simplicial complex.
\begin{Definition}\label{Def-CoverIdeal-SRIdeal} Let $\Delta$ be a simplicial complex on $[n]$.\\
	(1) The {\em cover ideal} of $\Delta$ is the monomial ideal 
	$$J(\Delta) := \bigcap_{F \in \mathcal{F}(\Delta)} \p_F.$$  
	(2) The {\em Stanley--Reisner} ideal of $\Delta$ is the monomial ideal 
	$$I_\Delta := \bigcap_{F \in \mathcal{F}(\Delta)} \p_{[n] - F}.$$ 
\end{Definition}

Next, we establish a few pieces of notation relative to monomial ideals used throughout the paper.
\begin{Notation}
	\begin{enumerate}
		\item Let $I$ be a monomial ideal in $R$. We denote by \begin{enumerate}[(i)]
			
			\item $G(I)$ the unique minimal generating set of $I$ consisting of monomials;
			\item ${\rm Ass}(R/I)$ the set of prime ideals $\p\subseteq R$ such that $\p=I:x$ for some $x\in R$; it is well--known that any such $\p$ has the form $\p=\p_F$ for some $F\subseteq [n]$;
			\item $\h\,I$ the {\em height} of $I$, which is $\h\,I=\min\{|F|\,\mid\,\p_F\in \Ass(R/I)\}$; 
			\item $\dim(R/I)$ the {\em dimension} of $R/I$, which is $n-\h\,I$;
			\item \sfp(I) the ideal of all squarefree monomials in $I$. 
		\end{enumerate}
		\item If $m$ is any monomial in $R$, we denote the {\em support} of $m$ by $\supp{m}$, i.e. 
		$$\supp{m}:=\{x_i\,\mid\,x_i \text{ divides }m\}.$$ 
		To simplify the notation, we will identify $\supp{m}$ with $\{i\in [n]\,\mid\,x_i\in \supp{m}\}$.
		\item For a subset of vertices $A \sub [n]$, we set $x_A:=\prod_{i \in A} x_i$, i.e. $x_A$ is the squarefree monomial  whose support is $A$.
\end{enumerate} \end{Notation}

E.g. if $I=(x,y)^2\cap (x,z)^2\cap (y,z)^2$, then $G(I)=\{xyz, x^2y^2, x^2z^2, y^2z^2\}$, $\Ass(R/I)$ $=\{(x,y),\,(x,z),\,(y,z)\}$, $\h\,I=2$, $\sfp(I)=(xyz)$ and $\supp{xyz}=\{x,y,z\}$.
\medskip

Let $\Delta$ be a matroid, we now recall some well--known connections between generators of ideals associated to a matroid and the circuits and hyperplanes of the matroid. 

\begin{Proposition}\label{Basic-Matroid-Properties} 
	Let $\Delta$ be a matroid. Then 
	\begin{enumerate}
		\item $\Delta\dual$ is a matroid. 
		\item $\{ \supp{m} : m \in G(I_{\Delta}) \}$ is the set of circuits of $\Delta$.
		\item $\{ \supp{m} : m \in G(J(\Delta)) \}$ is the set of cocircuits of $\Delta$.
		\item $H$ is a hyperplane of $\Delta$ if and only if $x_{[n] - H}$ is a minimal generator of $J(\Delta)$.
		\item $\dim(R/I)=r(\Delta)$, while $\h\,I_\Delta=\h\,J(\Delta^*)=r(\Delta^*)$. 
		\item $\Delta$ has a loop $\Llra$ $I_\Delta$ contains a variable $\Llra$ $J(\Delta)$ is extended from a smaller polynomial ring.
	\end{enumerate}
\end{Proposition}

We now define {\em matroidal} ideals in our context. For a squarefree monomial ideal $I$ there are, in the literature, two main (different) ways of defining $I$ to be a ``matroidal ideal": $G(I)$ could be generated by the circuits of a matroid, see e.g. \cite{NPS}, \cite{MT}, \cite{Var}, \cite{GHMN}, or $G(I)$ could be generated by the bases of a matroid, see e.g. \cite{HeHi}. 
To avoid confusion, we recall both definitions and we distinguish them by adding a $C$- or $B$- in front of ``matroidal".

\begin{Definition}\label{Def-Matroidal}
	A squarefree monomial ideal $I\subseteq R$ is 
	\begin{enumerate}
		\item {\rm \Cmatroid} if $I$ satisfies one of the following equivalent conditions:
		\begin{enumerate}[(i)]
			\item the elements of $G(I)$ satisfy the \underline{C}ircuits axioms of a matroid,
			\item $R/I^{(\ell)}$ is Cohen--Macaulay for some $\ell\geq 3$,
			\item  $I$ is the Stanley--Reisner ideal of a matroid,
			\item  $I$ is the cover ideal of a matroid.
		\end{enumerate}
		\item {\em \Bmatroid } if $I$ satisfies one of the following equivalent conditions:
		\begin{enumerate}[(i)]
			\item the elements of $G(I)$ satisfy the \underline{B}ases axioms of a matroid (i.e. $I$ is the facet ideal of a matroid), 
			\item $I$ is the Alexander dual of a \Cmatroid ideal.
		\end{enumerate}
	\end{enumerate}
\end{Definition} 

In this paper we only work with \Cmatroid ideals. 

\begin{Remark}\label{SRdual}
	In Definition \ref{Def-Matroidal}, the equivalence (i) $\Llra$ (iii) holds by definition of Stanley--Reisner ideal and Proposition \ref{Basic-Matroid-Properties} (2). Also, (ii) $\Llra$ (iii) holds by Theorem \ref{symbintro}. Finally, (iii) $\Llra$ (iv) follows by Proposition \ref{Basic-Matroid-Properties} (1) and Definition \ref{Def-CoverIdeal-SRIdeal}, in particular one has
	$$J(\Delta\dual) = I_{\Delta}\qquad \text{ and }\qquad J(\Delta) = I_{\Delta\dual}.$$
	
	Most statements will assume that $I$ is a \Cmatroid ideal, i.e. either $I=I_\Delta$ or $I=J(\Delta)$ for some matroid $\Delta$. However, by (iii) $\Llra$ (iv), our proofs will only involve only of one of these ideals and, whenever needed, we will describe a way to obtain a proof for the other ideal.
	
\end{Remark}

\section{The Structure Theorem for Symbolic Powers of \Cmatroid ideals and new characterizations of matroids}

In this section we prove our first two main results. The first one provides an explicit description of the {\em minimal monomial generating set} $G(\symp{I_{\Delta}})$ for the symbolic power $\symp{I_{\Delta}}$ of any Stanley--Reisner or cover ideal of a matroid $\Delta$. These minimal generators have a very specific ``tower" structure, resembling the structure in standard monomial theory. 

At the end of the section we prove our second main result, i.e. for a squarefree monomial ideal $I$, being \Cmatroid is actually equivalent to $G(\symp{I})$ having this ``tower" structure for some $\ell\geq 2$.
\medskip

We begin by recalling the definition of symbolic powers. In the context of this paper it suffices to define them for a squarefree monomial ideal $I$. See for instance \cite{SymPowIdeals} for a more general and in-depth treatment of symbolic powers.
\begin{Definition} \label{Def-symp-sfp}
	Let $I\sub R$ be a squarefree monomial ideal. The {\em $\ell$-th symbolic power} of $I$ is $$\symp{I} = \bigcap_{\p\in \Ass(R/I)} \p^{\ell}.$$
	We write $\sfp_\ell(I):=\sfp(\symp{I})$ for the squarefree part of $\symp{I}$.
\end{Definition}

We recall a well--known membership criterion. 

\begin{Remark}\label{Rmk-Symbolic-Criterion} Let $\Delta$ be a simplicial complex, $I=I_\Delta\subseteq k[x_1,...,x_n]$ and $m=x_1^{a_1}\cdots x_n^{a_n}$ for $a_i\geq 0$. Then with $A = \{ a_i \mid a_i > 0 \}$, 
$$m \in \symp{I} \Llra \sum_{i \in [n] - F}a_i=\sum_{i \in A - F}a_i \geq \ell \textrm{ for all }F \in \mathcal{F}(\Delta)\Llra \sum_{i \in G} a_i \geq \ell\textrm{ for all }G \in \mathcal{F}(\Delta\dual).$$
 
	Equivalently, in the language of vertex cover algebras in \cite{HHT}, $m\in \symp{I}$ if and only if the exponents vector of $m$ forms an $\ell$-cover of $\Delta\dual$. In addition, $m$ is a minimal generator of $\symp{I}$ if and only if the exponent vector of $m$ is a basic $\ell$-cover of $\Delta\dual$.
\end{Remark}

The following easy observation is often used to prove minimality of some generators of $\symp{I}$. We state it for squarefree monomial ideals because of our context, but the same proof holds for any (definition of) symbolic power of any monomial ideal.
\begin{Remark}\label{Rmk-Minimal}
	Let $I$ be a squarefree monomial ideal, $m\in G(\symp{I})$. If there exist $0\leq d_1,d_2\leq \ell$ and monomials $m_1,m_2$ with $m_i\in I^{(d_i)}$ for $i=1,2$ with $d_1+d_2=\ell$, then $m_i\in G(I^{(d_i)})$ for $i=1,2$.
\end{Remark}

\begin{proof} 
	Assume by contradiction $\wdt{m}\in G(I^{(d_1)})$ properly divides $m_1$, then $\wdt{m}m_2\in I^{(d_1)}I^{(d_2)}$ $\subseteq I^{(\ell)}$ properly divides $m$, contradicting the minimality of $m$.\end{proof}

We now introduce two new properties associated to a simplicial complex $\Delta$ and prove that  any of them is equivalent to $\Delta$ being a matroid. 

\begin{Definition}\label{Def-Star-Property} 
	Let $\Delta$ be a simplicial complex on $[n]$. For any subset $A\subseteq [n]$, let  
	$$h_A:=\min\{|A-H|\,\mid\,H\in \F(\Delta)\} \qquad \text{ and }\qquad  c_A := \min \{ |A\cap H| \; |\; H \in \F(\Delta) \}.$$
	
	We say that $\Delta$ satisfies properties
	\begin{itemize}
		\item[$(\star)$]  if, for every subset $A\subseteq [n]$ and $F\in \F(\Delta)$, there exists $G\in \F(\Delta)$ with 
		$$
		|A- G|=h_A, \text{ and } A-G\subseteq A-F;
		$$
		\item[$(\star\dual)$] if, for every subset $A \subseteq [n]$ and $F \in \F(\Delta)$, there exists $G \in \F(\Delta)$ with $$|A\cap G| = c_A, \textrm{ and } A \cap G \subseteq A \cap F.$$
	\end{itemize}

\end{Definition}

From the equality $A - H = A \cap ([n] - H)$ it follows that the two properties are dual:
\begin{Remark}\label{Star-Properties-Are-Dual} A matroid $\Delta$ satisfies property $(\star)$ if and only if $\Delta\dual$ satisfies property $(\star\dual)$.
\end{Remark}

Somewhat surprisingly, these two properties provide new characterizations of matroids.
\begin{Theorem}\label{MatroidStar}
	 Let $\Delta$ be a simplicial complex. The following are equivalent:
	
	\begin{enumerate}
		\item $\Delta$ is a matroid,
		\item $\Delta$ satisfies property $(\star)$,
		\item $\Delta$ satisfies property $(\star\dual)$.
	\end{enumerate}
\end{Theorem}

\begin{proof}
	In view of Proposition \ref{Basic-Matroid-Properties} (1) and Remark \ref{Star-Properties-Are-Dual}, we only need to show $(1) \Longleftrightarrow (2)$.
	
	$(1) \Longrightarrow (2)$: 
	Let $F\in \F(\Delta)$ be a basis and $A \sub [n]$. If $h_A = |A-F|$, just take $G$ to be $F$. We may then assume $h_A < |A-F|$. Let $F'$ be a basis of $\Delta$ such that $h_A = |A-F'|$, then necessarily $|A\cap F'| > |A \cap F|$. As $A\cap F'$ and $A \cap F$ are independent sets of different sizes in a matroid, 
	then one can find an independent set $G' \supseteq A\cap F$ such that $|G'| = |A\cap F'|$. Now, let $G$ be a basis containing $G'$. Then $G \supseteq A \cap F$ and hence $A - G \sub A - F$. Now as $G \supseteq G'$, we have $|A - G| \leq |A -G'| = |A - F'| = h_A$. But then by minimality of $h_A$ we have $|A-G| = h_A$.
	
	$(2) \Longrightarrow (1)$: We first show that $\Delta$ is pure. We will apply property $(\star)$ with $A = [n]$. Now, $$h_A = \min\{ |[n] - H| \; | \; H \in \F(\Delta)\}.$$ Suppose $\Delta$ is not pure, then there is $F \in \F(\Delta)$ such that $|[n]-F| > h_A$, then by $(\star)$ there is a $G \in \F(\Delta)$ such that $h_A= |[n]-G| < |[n]-F|$ and $[n]-G\subseteq [n]-F$. This implies that $F\subsetneq G$, contradicting that $F$ is a facet of $\Delta$.
	
	We next show that $\F(\Delta)$ satisfies Definition \ref{Def-Matroid-Basis}$(1)$. Since $\Delta$ is pure it suffices to show that, for any $F\neq G \in \F(\Delta)$ and $i \in F - G$, there exists $H\neq F$ in $\F(\Delta)$ such that $ F-i\subseteq H$ and $(G-F)\cap H \neq \emptyset$. Since $\Delta$ is pure then $|F-G|=|G-F|$. Set $A := (F \cup G) - i$. Notice that $A - F = G - F$ and $A - G = F - (G \cup i)$, hence $|A-F| > |A-G| \geq h_A$. By property $(\star)$ there exists $H \in \F(\Delta)$ with 
	$A- H \subsetneq A - F$, in particular this implies that $H\neq F$. 
	Since $(F-i)$, $(G-F)$ is a partition of $A$, and $A- H \subsetneq A - F = G - F$, we get both $F-i \subseteq H$ and $(G-F) \cap H \neq \emptyset$. 
\end{proof}

We are now ready for the first main result, the Structure Theorem. 

\begin{Theorem}(Structure Theorem)\label{MatroidSymPowerThm} Let $I$ be a \Cmatroid ideal (see Definition \ref{Def-Matroidal}). Then
			{\small \begin{center}
					$G(\symp{I}) = \bigg\{$
					\begin{tabular}{l|l} 
						& $m_i \in G(I^{(c_i)})$ and is squarefree, where $1 \leq c_i \leq \h\, I$\\
						$m=m_1 \cdots m_s$ 	 &\\
						&  with $\sum c_i = \ell$ and $\supp{m_1} \supseteq \ldots\supseteq \supp{m_s}$
					\end{tabular}
					$\bigg\}$.
			\end{center}} 
		\end{Theorem}
		
		\begin{Definition}\label{Def-SymbType}
			If $I$ is \Cmatroid and $m\in G(\symp{I})$, it is easily seen that there is a unique way of writing $m$ as $m=m_1\cdots m_s$ for $m_i\in \sfp_{c_i}(I)$ with $\sum c_i=\ell$ and $\supp{m_i}\supseteq \supp{m_{i+1}}$.
			
			We call the sequence of integers $(c_1,c_2,\ldots,c_s)$ the {\em symbolic type} of $m$.
		\end{Definition}

		\begin{proof} Since $J(\Delta) = I_{\Delta\dual}$, we only need to prove the statement for $I := I_{\Delta}$.
Throughout the proof, we use the equivalence of statements (1) and (2) in Theorem \ref{MatroidStar}.

			``$\subseteq$" 
			We proceed by induction on $\ell\geq 1$. For $\ell=1$ the statement is trivial. Now let $\ell > 1$ and let $m\in G(\symp{I})$. If $m\in \sfp_\ell(I)$, then there is nothing to show. 
			We may then assume $m$ is not squarefree. 
			Let $m_1$ be the squarefree part of $m$, i.e. $m_1:=x_A$ where $A:=\supp{m}$, and let $m_2:=m/m_1$. Since $\supp{m_2} \subseteq \supp{m_1}$, if we show there exists $c_1$ with $m_1\in I^{(c_1)}$ and $m_2\in I^{(\ell - c_1)}$, then by Remark \ref{Rmk-Minimal} we obtain $m_1\in G(I^{(c_1)})$ and $m_2\in G(I^{(\ell - c_1)})$ and then we are done by induction, provided $1\leq c_1 \leq \min\{\h\,I,\;\ell-1\}$. (Note: in general, if $\Delta$ is not matroidal, the inclusion $m_1\in G(I^{(c_1)})$ does not imply that $m_2=m/m_1\in I^{(\ell-c_1)}$.)

			Let $c_1:=h_A$ be the number defined in property $(\star)$ of Definition \ref{Def-Star-Property}. By the criterion in Remark \ref{Rmk-Symbolic-Criterion}, we see that $c_1 = \max \; \{ t \; | \; m_1 \in I^{(t)} \}$, so $m_1 \in I^{(c_1)}$. Since $m_1 \in \sqrt{I}=I$, then $c_1\geq 1$, and since $m$ is not squarefree, then $m \neq m_1$, thus $m_1 \notin \symp{I}$, so $c_1\leq \ell-1$.  Also, since $m_1$ divides $x_{[n]}=x_1\cdots x_n$ and, clearly, $x_{[n]}\in I^{(\h\,I)}-I^{(\h\,I+1)}$, then $c_1\leq \h\,I$, so $1\leq c_1 \leq \min\{\h\,I,\;\ell-1\}$. 
			
			Next we show $m_2 \in I^{(\ell-c_1)}$. Write $m:=\prod_{i\in A}x_i^{a_i}$ for $a_i\geq 0$, so $m_2:=\prod_{i\in A}x_i^{a_i-1}$. Let $B :=\supp{m_2} \subseteq A$, so $a_i=1$ $\Llra$ $i\in A-B$. For any $F\in \F(\Delta)$, then $\sum_{i \in B-F} (a_i-1)=\sum_{i \in A-F} (a_i-1)$, so by the criterion in Remark \ref{Rmk-Symbolic-Criterion}, it suffices to prove that, 
			$$
			\sum_{i \in A-F} (a_i-1)\geq \ell - c_1.
			$$ 
			By property $(\star)$, there is $G\in \F(\Delta)$ with  
			$|A-G|=c_1$ and $A-G\subseteq A-F$. Then  
			$$\sum_{i \in A-F} (a_i - 1) \geq \sum_{i \in A-G} (a_i - 1) = 
			\bigg(\sum_{i \in A-G} a_i \; \bigg) - |A - G| \geq \ell - c_1.$$

			``$\supseteq$" Again we induct on $\ell\geq 1$. If $\ell = 1$, there is nothing to show. For $\ell > 1$, let $m$ be a monomial of the form $m=m_1\cdots m_s$ for $m_i\in \sfp_{c_i}(I)$ with $\supp{m_i}\supseteq \supp{m_{i+1}}$ for every $i$, and assume $1\leq c_i \leq \h\, I$ and $\sum_{i=1}^{s}c_i= \ell$. It is clear that $m \in \symp{I}$, hence $m$ is divisible by some $\wdt{m}\in G(\symp{I})$, we will show $m = \wdt{m}$. By the forward inclusion, $\wdt{m} = \wdt{m}_1 \cdots \wdt{m}_t$  with $\wdt{m}_j \in \sfp_{d_j}(I)$ for some $d_j\geq 1$ and $\supp{\wdt{m}_j}\supseteq \supp{\wdt{m}_{j+1}}$ for every $j$. Now $\supp{\wdt{m}} \subseteq \supp{m}$. If their supports are equal then $\wdt{m}_1 = m_1$, and $(\wdt{m}/\wdt{m}_1) \,|\, (m/m_1)$ in $I^{(\ell - c_1)}$, hence by induction $\wdt{m}/\wdt{m}_1 = m/m_1$, thus $\wdt{m} = m$. 
			
			Hence we may assume $\supp{\wdt{m}}\subsetneq \supp{m}$. Since $\wdt{m}_1 \, | \, m_1$ and they are both squarefree, then $d_1 < c_1$. This yields a contradiction, because, by induction, $\wdt{m}/\wdt{m}_1$ is a minimal generator of $\ell - d_i$ and divides $m/m_1$, which is a minimal generator of the smaller symbolic power $\ell - c_i$.
		\end{proof}
		
		It is natural to ask whether there are other squarefree monomial ideals $I$ whose generators $G(\symp{I})$ have such a ``tower" structure -- i.e. each layer is a squarefree minimal generator of some $I^{(c_i)}$ with $\sum c_i =\ell$, and their supports are nested. In Theorem \ref{Thm-Matroid-Equiv} we show that the answer is no, i.e. \Cmatroid ideals are the only ones. 
		
		It follows immediately from the Structure Theorem that each symbolic power $I^{(\ell)}$ can be described only in terms of squarefree parts. 		Recall that $\sfp_{\ell}(I)$ denotes the squarefree part of $\symp{I}$.
		
		\begin{Theorem}\label{SymPowerSum} Let $\ell\in \ZZ_+$. Let $I$ be a \Cmatroid ideal, and write $c:=\h\,I$. Then
			$$\symp{I} = \sum_{i = 0}^{\ell}I^{(i)}I^{(\ell - i)} + \sfp_{\ell}(I), \quad \text{ and }\quad \symp{I} = \sum_{\substack{a_1 + 2a_2\; ... \; + ca_{c}=\ell \\ a_1,\ldots,a_c\geq 0}}\sfp_{1}(I)^{a_1}\cdots \sfp_{c}(I)^{a_c}.$$
		\end{Theorem}
		
		
	
	\begin{Example} Consider the \Cmatroidal ideal $I = (af,cd,bde,bce) \sub k[a,b,c,d,e,f]$, which has height $3$. To illustrate the use of Structure Theorem \ref{MatroidSymPowerThm}, we write a couple of minimal generators of some symbolic powers of $I$. Using Proposition \ref{SquareFreePartIsSkeleton} or Corollary \ref{SqFree-Skeleton-Elongation}, we find that 
		$$\sfp_2(I) = (acdf, bcde, abcef, abdef), \quad \quad \sfp_3(I) = (abcdef).$$ 
		
		To obtain minimal generators of certain symbolic powers, we multiply powers of minimal generators of $I,$  $\sfp_2(I),$ and $\sfp_3(I)$ so that the supports of the generators form a containment chain. For instance, if we take $m_1 = abcdef \in \sfp_3(I)$, $m_2 = acdf \in \sfp_2(I)$, $m_3 = af \in I,$ we get that $$(m_1)^2(m_2)^4(m_3)^1 = a^7b^2c^6d^6e^2f^7 \in G(I^{(2 \cdot 3 + 4\cdot 2 + 1 \cdot 1)}) = G(I^{(15)}).$$ 
		
		Conversely, by the above theorem, all monomials in $G(\symp{I})$ can be decomposed as above, e.g. $m =ab^3c^6d^6e^3f\in G(I^{(10)})$, can be written as $m = (abcdef)(bcde)^2(cd)^3$.
		
		One way to visualize these minimal generators is to arrange them in ``towers'', e.g. 
		$$(abcdef)(bcde)^2(cd) = \begin{array}{l}
			\;\,\;\;cd   \\
			\; \;bcde  \\
			\;\;bcde   \\
			abcdef
		\end{array} \in I^{(8)}.$$

	\end{Example}
	
	Using the methods introduced by Geramita, Harbourne, Migliore, and Nagel in \cite{GHMN}, one can specialize the Structure Theorem to any symbolic power of any specialization of a \Cmatroid ideal. We recall the relevant definitions. 
	
	Recall that homogeneous polynomials $f_1,\ldots,f_r$ in $R$ form a {\em regular sequence} if every $f_{i}$ is a non--zero divisor on $S/(f_1,\ldots,f_{i-1})$. 
	\begin{Definition}\label{Def-Matroidalconfig}(see \cite[Section 3]{GHMN})
		Let  $f_1,\ldots,f_n$ be homogeneous polynomials (of possibly different degrees) in some polynomial ring $S$ over $k$. Let $\varphi:R\lra S$ be the $k$-algebra homomorphism defined by $\varphi(x_i)=f_i$ for every $i=1,\ldots,n$. Let $L\subseteq R$ be any monomial ideal, we write $L_*$ for the ideal generated by $\varphi(G(L))$.
		\begin{enumerate}
			\item Let $I$ be a \Cmatroid ideal in $R$, let $c:=\h\,I$. Further assume assume that any $c+1$ of the $f_i$'s form a regular sequence. Then
			\begin{itemize}
				\item The ideal $I_*$  
				is called a {\em specialization} of $I$, or the (defining) {\em ideal of a matroidal configuration (of hypersurfaces)}.
				\item If, additionally, all $f_i$'s are linear forms, $I_*$ is called a {\em linear} specialization of $I$;
				\item If $I_*$ is a linear specialization of $I$ and $\dim(S)=c+1$, then $I_*$ is called the (defining) {\em ideal of a matroidal configuration of points}. (because it is a radical ideal whose variety is a set of points in $\mathbb P^c$.)
			\end{itemize}
			\item Moreover an element $M$ of the form $M=f_1^{b_1}\cdots f_n^{b_n}$ for some $b_i\geq 0$ is called a {\em monomial} in $f_1,\ldots,f_n$. $M$ is a {\em squarefree monomial} if, additionally, $b_i\leq 1$ for all $i$. As for ``usual" monomials, one defines $\supp{M}=\{f_i\,\mid\,b_i>0\}$. Since $c\geq 1$, by \cite[Prop~3.8]{M} the set $\supp{M}$ is well-defined.
			
		\end{enumerate}

		\end{Definition}

		An example of a matroidal configuration is a {\em star configuration}, which is a specialization of a uniform matroid, e.g. see \cite{GHMstar}.
		
		While it is immediately seen that $(L_*)^\ell = (L^\ell)_*$ is true for any $\ell\geq 1$ and any monomial ideal $L$, the analogous equality $\symp{(L_*)} = (\symp{L})_*$ is, in general, false. However, in \cite[Thm~3.6(1)]{GHMN} it is proved that $\symp{(I_*)} = (\symp{I})_*$ does hold for any \Cmatroid ideal $I$, in particular $\symp{(I_*)}$ is generated by monomials in the forms $f_1,\ldots,f_n$, more precisely by the set $\varphi(G(\symp{I}))$, which we (slightly improperly) denote $G(\symp{I_*})$. We then have the following structure theorem: 
		\begin{Corollary}
			Let $I_*$ be the ideal of a matroidal configuration of hypersurfaces. . 
			Then
			{\small \begin{center}
					$G(\symp{I_*}) = \Bigg\{$
					\begin{tabular}{l|l} 
						& $M_i \in G(I_*^{(c_i)})$ and is squarefree, where $1 \leq c_i \leq \h\,I$\\
						$M=M_1 \cdots M_s$ 	 &\\
						&  with $\sum c_i = \ell$ and $\supp{M_1} \supseteq \ldots\supseteq \supp{M_s}$
					\end{tabular}
					$\Bigg\}$.
			\end{center}}		
		\end{Corollary}

		If $I$ is a \Cmatroid ideal $I$ then, by the Structure Theorem \ref{MatroidSymPowerThm}, $G(\symp{I})$ can be described as long as we understand $\sfp_\ell(I)$, i.e.  the ideal of squarefree monomials in $\symp{I}$, for $1 \leq \ell \leq \h\, I$. To make Theorem \ref{MatroidSymPowerThm} more effective, we next describe $G(\sfp_\ell(I))$. 
		
		This first proposition applies more generally to any squarefree monomial ideal. 
		
		\begin{Proposition}\label{SquareFree-Are-LCM} Let $I$ be a squarefree monomial ideal. For any $m\in G(\sfp_\ell(I))$ there exist $m_1,...,m_\ell \in G(I)$ such that $\LCM(m_1,...,m_\ell) \,|\, m$ and  $m_i$ does not divide $\LCM(m_1,...,m_{i-1})$ for $2 \leq i \leq \ell$. 
		\end{Proposition}
		
		\begin{proof}  Let $\Delta$ be a simplicial complex such that $I=J(\Delta)$. Considering $m$ as a basic $\ell$-cover of $\Delta$ we see that there exists $F\in \F(\Delta)$ such that $|\supp{m} \cap F| = \ell$. Write $\supp{m} \cap F = \{ y_1,...,y_\ell \}$. Now for each $1 \leq i \leq \ell$, consider $$\wdt{m}_i = \frac{my_i}{y_1\cdots y_\ell}.$$
			
			Since we have removed $\ell-1$ elements from $\supp{m}$ to obtain $m_i$, each $\wdt{m}_i$ is a 1-cover. Then for each $i$ there exist a $m_i \in G(I)$ so that $m_i \,|\, \wdt{m}_i$. Now for all $i$, since $\supp{\wdt{m}_i}\cap F=\{y_i\}$, then necessarily $y_i \,|\, m_i$, as otherwise $m_i$ would not cover $F$. Clearly, $\LCM(m_1,...,m_\ell) \,|\, \LCM(\wdt{m}_1,...,\wdt{m}_\ell) = m$ and, for each $i$, $y_i \, | \, m_i$ but $y_i \nmid m_j$ for $j \neq i$. Hence $m_1,...,m_\ell$ has the desired $\LCM$ property in the statement.
		\end{proof}
		
		We next show that the converse holds when $I$ is a \Cmatroid ideal. In this way, 
		$\sfp_\ell(I)$ is generated precisely by the $\LCM$ of $\ell$ many minimal generators of $I$.
		
		\begin{Proposition}\label{Matroid-LCM} Let $I$ be \Cmatroid. For $\ell \leq \h\,I$, let $m_1,...,m_\ell$ be in $G(I)$ with the property that for $2 \leq i \leq \ell$, $m_i$ does not divide $\LCM(m_1,...,m_{i-1})$. Then $\LCM(m_1,...,m_\ell) \in \symp{I}$.
		\end{Proposition}
		
		\begin{proof}
			We induct on $\ell \geq 2$. When $\ell = 2$, as $m_1m_2 \in I^2\subseteq I^{(2)}$, there is a minimal squarefree generator $\wdt{m} \in I^{(2)}$ dividing $m_1m_2=\LCM(m_1,m_2)\GCD(m_1,m_2)$. 
			We observe that $\wdt{m}$ is squarefree because, if not, by the Structure Theorem \ref{MatroidSymPowerThm}, $\wdt{m}=m_0^2$ for some $m_0\in G(I)$, thus $m_0$ divides $\GCD(m_1,m_2)$ contradicting the minimality of $m_1$ and $m_2$.	
			Then, $\wdt{m}$ is squarefree, thus it divides $\supp{m_1m_2}=\LCM(m_1,m_2)$, yielding $\LCM(m_1,m_2)\in I^{(2)}$. 
			
			The case $\ell>2$ follows in a similar manner. Let $M = \LCM(m_1,...,m_{\ell-1})$ and note that $M$ is squarefree. By induction, $M\in I^{(\ell-1)}$. Now $Mm_\ell\in \symp{I}$, and we write $$Mm_\ell=\LCM(M,m_\ell)\GCD(M, m_\ell).$$ As before, there is a $\wdt{m} \in G(\symp{I})$ dividing $Mm_\ell$ and it suffice to prove that $\wdt{m}$ is squarefree. Suppose not, since $x_j^3\nmid Mm_\ell$ for any $j$, then, by Theorem \ref{MatroidSymPowerThm},  $\wdt{m} = \wdt{m}_1\wdt{m}_2$ for squarefree monomials $\wdt{m}_i$ with $\supp{\wdt{m}_2}\subseteq \supp{\wdt{m}_1}$. Then $\wdt{m}_2^2\,\mid\, Mm_{\ell}=\LCM(M,m_\ell)\GCD(M, m_\ell)$, so $\wdt{m}_2\,\mid\,\GCD(M,m_{\ell})$. By assumption, $m_\ell \nmid M$ so $\GCD(M, m_\ell)$ properly divides $m_\ell$ and then $\wdt{m}_2$ properly divides $m_\ell$, contradicting the minimality of $m_\ell$. 
		\end{proof}
		
		\begin{Corollary}\label{Corr-Matroid-LCM} Let $I$ be a \Cmatroid ideal. Then the squarefree part of $\symp{I}$ is
			{\small $$\sfp_\ell(I)=(\LCM(m_1,...,m_\ell) \, | \, m_1,...,m_\ell \in G(I), \textrm{ for } 2 \leq i \leq \ell  \; m_i \nmid \; \LCM(m_1,...,m_{i-1})),$$}
			and 		$G(\sfp_\ell(I))$ consists of the minimal elements with respect to divisibility of the displayed set. 
			
			In particular if $m \in \sfp_\ell(I)$ and $\wdh{m} \in \sfp_k(I)$  with $\supp{m} \neq \supp{\wdh{m}}$, then $\LCM(m,\wdh{m}) \in \sfp_{\ell+1}(I)$.
		\end{Corollary} 
		
		\begin{proof} The ``$\supseteq$" inclusion follows from Proposition \ref{Matroid-LCM}. For the ``$\sub$" inclusion, let $m \in G(\sfp_\ell(I))$, then by Proposition \ref{SquareFree-Are-LCM}, we get $m_1,...,m_\ell \in G(I)$, with the non-divisibility property as in the statement, and such that $\LCM(m_1,...,m_\ell) \,|\, m$. But then by Proposition \ref{Matroid-LCM}, $\LCM(m_1,...,m_\ell) \in \sfp_\ell(I)$, and hence equals $m$ by minimality of $m$.
  
  Now we show the ``in particular" part of the statement. From the first part we can write $\wdh{m} = \LCM(\wdh{m}_1,...,\wdh{m}_k)$ with each $\wdh{m}_i \in G(I)$ and $m = \LCM(m_1,...,m_\ell)$ with each $m_i \in G(I)$. As $\supp{\wdh{m}} \neq \supp{m}$, there exist a $\wdh{m}_i$ such that $\wdh{m}_i \nmid m$, hence by Proposition \ref{Matroid-LCM}, $\LCM(m_1,...,m_\ell,\wdh{m}_i)\in \sfp_{\ell+1}(I)$ and it divides $\LCM(m,\wdh{m})$.
		\end{proof}
		
		\begin{Corollary}\label{Cor-Disjoint}
			Let $I$ be \Cmatroid, $m=m_1\cdots m_\ell$ for some $m_i\in G(I)$. If there are $i,j$ such that $m_i\neq m_j$ and $\supp{m_i}\cap \supp{m_j}\neq\emptyset$, then $m\notin G(\symp{I})$.
		\end{Corollary}
		
		\begin{proof}
			By Proposition \ref{Matroid-LCM},  $\LCM(m_i,m_j)\in I^{(2)}$ and properly divides $m_im_j$, so $m_im_j$ is not a minimal generator of $I^{(2)}$. Now apply Remark \ref{Rmk-Minimal} to $m=(m_im_j)(m/m_im_j)$.
		\end{proof}		
		To state the next result, we introduce the following notion.  
		\begin{Definition}\label{Def-Skeleton-Constructions} 
			Let $\Delta$ be a simplicial complex. 	For any $\ell\geq 1$, we set $$\Delta_\ell:=\langle F-A\,\mid\,F\in \F(\Delta),\,A\subseteq F,\,|A|=\ell-1\rangle.$$  
		\end{Definition}
		
		\begin{Remark}
			If $\Delta$ is pure, then $\Delta_{\ell}$ is the $(\dim \Delta - \ell + 1)$-skeleton of $\Delta$. If $\Delta$ is a  matroid, then $\Delta_{\ell}$ is precisely the $(\ell-1)$-th truncation $\Delta_{[\ell-1]}$ of $\Delta$. 
		\end{Remark}
		
		We are now ready to give an alternative description of the squarefree part $\sfp_{\ell}(J)$ of $\symp{J}$. This description holds for any squarefree ideal $J$. 
		
		\begin{Proposition}\label{SquareFreePartIsSkeleton} Let $J$ be a squarefree monomial ideal.
			Then  $$\sfp_{\ell}(J)=J(\Delta_{\ell}),$$ where $\Delta$ is a simplicial complex such that $J = J(\Delta)$.  
		\end{Proposition}
		
		\begin{proof} 
			It suffices to prove $$\sfp_{\ell}(J) = \bigcap_{\p \in \Ass(R/J)} \sfp_{\ell}(\p) = \bigcap_{\p=\p_F \in \Ass(R/J)}\bigcap_{\substack{A \sub F \\ |A| = \ell - 1}}(\p_{(F-A)}).$$ 
			Since, clearly, $\sfp(I_1\cap I_2)=\sfp(I_1)\cap \sfp(I_2)$	holds for any two monomial ideals $I_1,I_2$, then the first equality follows from the definition of $\symp{J}$. The second equality follows because every $\p\in \Ass(R/J)$ has the form $\p=(x_{i_1},\ldots,x_{i_c})$ and it easily seen that $$\sfp(\p^\ell)=\bigcap_{\substack{A \sub \{i_1,\ldots,i_c\} \\ |A| = \ell - 1}}(\p_{(\{i_1,\ldots,i_c\} - A)}).$$\end{proof}
   When $I$ is \Cmatroid, Corollary \ref{Corr-Matroid-LCM} gives $G(\sfp_\ell(I))$, while  Proposition \ref{SquareFreePartIsSkeleton} gives the associated primes of $\sfp_\ell(I)$.
		\begin{Example} $(1)$ Consider the ideal $I = (ae,af,bdf,be,cdf,ce) = (a,b,c) \cap (a,d,e) \cap (e,f)$. By Proposition \ref{SquareFreePartIsSkeleton}, to compute $\sfp_2(I)$ we first remove one variable (in all possible ways) from each associated prime of $I$, and then compute the intersection; 
			$$\sfp_2(I) = \Big( (a,b) \cap(b,c) \cap (a,c) \Big) \bigcap \Big( (a,d) \cap(a,e) \cap (d,e) \Big) \bigcap \Big( (e) \cap(f) \Big)= (abef,acef,bcdef).$$
			
			$(2)$ Consider the ideal $I = (abe, ace, ad, bc, bde, cde)$. Now, $|\Ass(R/I)| = 8$. Using Proposition \ref{SquareFreePartIsSkeleton}, to compute $\sfp_2(I)$ we need to intersect $8{3 \choose 1} = 24$ primes. Alternatively, since $I$ is \Cmatroidal, we can use Corollary \ref{Corr-Matroid-LCM}, and compute the ideal of pairwise $\LCM$ of the generators of $I$. This yields $$\sfp_2(I) = (abcd,abce,abde,acde,bcde).$$
		\end{Example}
		
		\begin{Corollary}\label{SqFree-Skeleton-Elongation} Let $\Delta$ be a matroid of rank $c = r(\Delta)$ and corank $d = r(\Delta\dual)$. Let $J$ be the cover ideal of $\Delta$ and $I$ be the Stanley--Resiner ideal of $\Delta$. Note that $c = \h\, J$ and $d = \h\, I$. Then 
			\begin{enumerate}
				\item For $1 \leq \ell \leq c\,$, $\sfp_{\ell}(J)$ is the cover ideal of $\Delta\trunc{\ell-1}$, the $(\ell-1)$-truncation of $\Delta$. 
				\item For $1 \leq \ell \leq d\,$, $\sfp_{\ell}(I)$ is the Stanley--Reisner ideal of $\Delta\elong{\ell-1}$, the $(\ell - 1)$-elongation of $\Delta$.
			\end{enumerate}
			
			\begin{proof} $(1)$ follows from Proposition \ref{SquareFreePartIsSkeleton}, and 
				$(2)$ follows from $(1)$ since forming the $\ell$-elongation is dual to forming the $\ell$-th truncation. 
			\end{proof}
			
		\end{Corollary}
		
		The minimal generators of $\symp{J(\Delta)}$ can be viewed as basic $\ell$-covers of a simplicial complex $\Delta$. Next, we show another combinatorial connection; if $\Delta$ is matroidal, then the minimal squarefree generators of $\symp{J(\Delta)}$ correspond to complements of flats of $\Delta$. Recall that $x_A=\prod_{i\in A}x_i$.
		
		\begin{Proposition}\label{SquareFree-Correspond-To-Flats} Let $\Delta$ be a matroid on $[n]$.  Then $H$ is a flat of $\Delta$ of rank $r(\Delta) - \ell$ if and only if $x_{[n] - H} \in G(\sfp_\ell(J(\Delta)))$. 
		\end{Proposition}
		\begin{proof} 
			$H$ is a flat of $\Delta$ of rank $r(\Delta) - \ell$ $\Longleftrightarrow$ $|H \cap F| = r(\Delta) - \ell$ for all $F \in \mathcal{F}(\Delta)$, and for any vertex $i \notin H$, $\exists\,F \in \mathcal{F}(\Delta)$ such that $|H \cap F| = r(\Delta) - \ell + 1$ $\Longleftrightarrow$ $|([n] - H) \cap F| = \ell$ for all $F \in \mathcal{F}(\Delta)$, and for any $i \in [n] - H$ $\exists\,F \in \mathcal{F}(\Delta)$ such that $|([n] - (H \cup i)) \cap F| = \ell - 1$ $\Longleftrightarrow$ $x_{[n] - H}$ is a basic $\ell$-cover (see \cite{HHT}) $\Longleftrightarrow$ $x_{[n]-H}$ is a minimal squarefree generator of $\symp{J(\Delta)}$.
		\end{proof}
		
		Combining Propositions \ref{Matroid-LCM} and \ref{SquareFree-Correspond-To-Flats} and passing to the complement, we recover the well--known result that every flat of a matroid $\Delta$ of rank $r(\Delta) - \ell$ is an intersection of $\ell$ many hyperplanes.

		\begin{Remark}\label{Rmk-Circuit-Axioms-LCM}

			By Definition \ref{Def-Matroidal}, a squarefree monomial ideal
			$J$ is \Cmatroid if and only if for any pair $m_1 \neq m_2\in G(J)$ and any variable $x$ dividing both $m_1$ and $m_2$, there exists $m_3\in G(J)$ dividing $\LCM(m_1,m_2)/x$. 
		\end{Remark}

		We are now ready to prove the second main result of this section. It states that  if  $G(I^{(\ell)})$ has the shape described in the Structure Theorem for {\em some} symbolic power $\ell\geq 2$, then $I$ is $C$-matroidal. 
		In particular, one only needs to look at $G(I^{(2)})$ to determine whether $I$ is $C$-matroidal. 
		It is complementary to the main results of \cite{MT}, \cite{TT} and \cite{Var}, which requires checking the Cohen--Macaulay property of some symbolic power $\symp{I}$ with $\ell\geq 3$.
		
		For a monomial ideal $L$ with $G(L)=\{m_1,\ldots,m_s\}$, we write $L^{[2]}:=(m_1^2,\ldots,m_s^2)$. 
		\begin{Theorem}\label{Thm-Matroid-Equiv}(Characterization of matroids in terms of generators of some symbolic power) 
			Let $\Delta$ be a simplicial complex. Let $I$ be either the Stanley Reisner or cover ideal of $\Delta$. The following are equivalent:
			
			\begin{enumerate}[(i)]
				\item\label{TFAE-Mat-1} $\Delta$ is a matroid. 
					\item For all $\ell \geq 1$, the elements of $G(\symp{I})$ have the form described in Theorem $\ref{MatroidSymPowerThm}$.
					\item For some $\ell \geq 2$, the elements of $G(\symp{I})$ have the form described in Theorem $\ref{MatroidSymPowerThm}$.
				\item $G(I^{(2)}) = G(I^{[2]})\cup G(\sfp_2(I))$ 
				
			\end{enumerate}
			
		\end{Theorem}

		\begin{proof} Since $\Delta$ is a matroid if and only if $\Delta\dual$ is a matroid and $I_{\Delta} = J(\Delta\dual)$, it suffices to prove the theorem when $I=J(\Delta)$. 
			
			$(i) \Longrightarrow (ii)$ is Theorem $\ref{MatroidSymPowerThm}$. Next, we observe that $(iv) \Llra (iv')$ where $(iv')$ is the property that $G(I^{(2)})$ has the form described in Theorem \ref{MatroidSymPowerThm}. Indeed, the elements in $G(I^{(2)})$ have the form described in  Theorem \ref{MatroidSymPowerThm} if and only if they have symbolic type $(1,1)$, i.e. they are of the form $m^2$ for some $m\in G(I)$, or they have symbolic type $(2)$, i.e. they are in $G(\sfp_2(I))$. It follows that $G(I^{(2)})$ has the form of Theorem \ref{MatroidSymPowerThm} if and only if $G(I^{(2)}) =G(\sfp_2(I)) \cup G(I^{[2]})$, since no element of $G(I^{[2]})$ is divisible by an element of $G(\sfp_2(I))$, and vice versa.
			
			$(ii) \Lra (iv')$ is now obvious. $(iv)\Lra (iii)$ is clear too. $(iii) \Longrightarrow (i)$. 
			By Proposition \ref{Basic-Matroid-Properties}(3), $\Delta$ is a matroid $\Llra$ 
			the complements of $G(I)$ form the set of cocircuits of a matroid. We then check the circuit axiom, translated in terms of elements in $G(I)$ as described in Remark \ref{Rmk-Circuit-Axioms-LCM}.
			Let $m_1\neq m_2\in G(I)$ and $x\in \supp{m_1}\cap \supp{m_2}$, then
			$m_1^{\ell-1}m_2 \in I^\ell \subseteq \symp{I}$, and then there exists $M\in G(\symp{I})$ dividing $m_1^{\ell-1}m_2$. Let $M_1=\sqrt{M}=x_{\supp{M}}$. Clearly, $M_1$ divides $\wdt{m}:=\sqrt{m_1^{\ell-1}m_2}=\LCM(m_1,m_2)$.
			By assumption $(iii)$, $M_1$ is a minimal generator of $I^{(c_1)}$, where $c_1\geq 1$ is the first entry of the symbolic type $(c_1,\ldots,c_s)$ of $M$. We claim that $c_1>1$. Assume not, then $M_1 \in G(I)$, thus, by $(iii)$ and Theorem \ref{MatroidSymPowerThm}, the symbolic type of $M$ is $(1,1,\ldots,1)$, i.e. $M=M_1^\ell$. Since $M$ divides $m_1^{\ell-1}m_2$, then $M_1\,\mid\,\GCD(m_1,m_2)$ divides $m_1$ and $m_2$. Since $M_1,m_1,m_2\in G(I)$, then $M_1=m_1=m_2$, which is a contradiction because $m_1 \neq m_2$.  
				Hence $M_1\in \sfp_2(I)$ and then also $\wdt{m}\in \sfp_2(I)$, in particular, $\supp{\wdt{m}}$ is a $2$-cover of $\Delta$. It follows that $\supp{\wdt{m}} - \{x \}$ is a $1$-cover of $\Delta$. Hence we can refine this cover to a minimal $1$-cover. In other words, there is a minimal generator $m_3 \in I$, so that $\supp{m_3} \subseteq \supp{\wdt{m}} - x \subseteq ( \supp{m_1}  \cup \supp{m_2} ) - x$.
		\end{proof}
		
		\begin{Example} Consider the ideal $I = (ab, acd, ace, ade, bcd, bce, bde, cde) \sub k[a,b,c,$ $d,e]$ which has primary decomposition $$I = (a,b,c)\cap(a,b,d)\cap(a,b,e)\cap(a,c,d) \cap (a,c,e) \cap (a,d,e) \cap (b,c,d) \cap (b,c,e) \cap (b,d,e).$$ Using the primary decomposition we compute $I^{(2)}$ yielding $$I^{(2)} = (a^2b^2, a^2c^2d^2, a^2c^2e^2, a^2d^2e^2,b^2c^2d^2,b^2c^2e^2, b^2d^2e^2, c^2d^2e^2, abcd, abce,abde,acde,bcde).$$ 
			By Theorem \ref{Thm-Matroid-Equiv} (iv), since $I^{(2)}$ is generated by the squares of the generators of $I$ along with squarefree monomials, we see that $I$ is a \Cmatroidal ideal.
		\end{Example}

		\section{The structure of the Symbolic Rees Algebra and a combinatorial interpretation of the Noether Number}
		
		In this section we describe in details the symbolic Rees algebra of a \Cmatroid ideal $J$. For sake of clarity, we will prove all results when $J=J(\Delta)$ is the cover ideal of a matroid $\Delta$, but by passing to $\Delta^*$ and dualizing the properties, analogous statements can be obtained when $I=I_\Delta$ is the Stanley--Reisner ideal of $\Delta$. Now let $J$ be \Cmatroid\!. We will use the structure theorem to precisely describe the minimal algebra generators of $\R_s(J)$, and furthermore we will see that the symbolic Noether number of $J$ (see Definition \ref{Def-SymbNoeth}) is related to the ranks of the connected components of the underlying matroid.
		
		First, for an ideal $I$ in a Noetherian ring $S$, the symbolic Rees algebra, denoted $\R_s(I)$, is the following graded $S$-subalgebra  of $S[t]$, $$\R_s(I) = S[It, I^{(2)}t^2, I^{(3)}t^3,...] \sub S[t].$$
		
		According to \cite{HHT}, if $I$ is any monomial ideal in a polynomial ring, then $\R_s(I)$ is Noetherian. 
		Inspired by the Noether number in Invariant Theory, we suggest the following definition:
		\begin{Definition}\label{Def-SymbNoeth}
			Let $I$ be an ideal in a Noetherian ring $S$. If $\R_s(I)$ is Noetherian, we define the {\em symbolic Noether number of $I$} (or the {\em generation type} of $\R_s(I)$, see \cite{Bah}) as the largest $t$-degree of a minimal $S$-algebra generator of $\R_s(I)$.  
		\end{Definition}
		
		Herzog, Hibi, and Trung have proved that the symbolic Noether number of a squarefree monomial ideal $I\subseteq R$ can be arbitrarily large -- in fact, it cannot be bounded, in general, by any linear function in the number $n$ of variables of $R$ \cite[Example~5.5]{HHT}. In stark contrast with this result, the Structure Theorem implies that the symbolic Noether number of a \Cmatroid ideal is bounded above by its height, which, of course, is at most $n$.
			
			\begin{Proposition}\label{SymReesAlgDescription} Let $J \sub R$ be a \Cmatroid ideal of $\h\,J=c$.  
			Then $\R_s(J) = R[Jt, \sfp_2(J)t^2,...,$ $\sfp_{c}(J)t^{c}]$. 
			In particular, the symbolic Noether number of $J$ is at most $c$.
		\end{Proposition}
		
		\begin{proof} 
			It follows immediately from Theorem \ref{SymPowerSum}.
		\end{proof}
		
		Note that, however, with $J$ as above, the set $\{Jt, \sfp_2(J)t^2,...,\sfp_{c}(J)t^c\}$ need not be a set of minimal algebra generators. In particular the symbolic Noether number could be strictly less than $c$. Our next goal is to provide a {\em precise}, concrete description of the symbolic Noether number. We will relate it to the connected components of $\Delta$.
		
		\begin{Definition} Let $A_1,A_2$ be disjoint sets and $\Delta_i$ a matroid on $A_i$ for $i=1,2$. The {\em direct sum} $\Delta_1 \oplus \Delta_2$ is the matroid on $A_1 \cup A_2$ whose set of bases is $\{ F_1 \cup F_2 : F_i \in \mathcal{F}(\Delta_i) \text{ for }i=1,2.\}$.
			
			A matroid $\Delta$ is {\em disconnected} if $\Delta = \Delta_1 \oplus \Delta_2$ for submatroids $\Delta_1$ and $\Delta_2$, otherwise $\Delta$ is {\em connected}.
			If $\Delta = \Delta_1\oplus \cdots \oplus \Delta_s$ and all $\Delta_i$'s are connected, then we call the $\Delta_i$'s the {\em connected components} of $\Delta$. 
		\end{Definition}
		
		\begin{Definition} 
			Let $\Delta$ be a matroid on $[n]$. For any $A \subseteq [n]$, the {\em restriction} $\Delta|_A$ of $\Delta$ to $A$ is the matroid whose independent sets are the independent sets of $\Delta$ which are contained in $A$.
		\end{Definition}
		
		Clearly, if a matroid $\Delta$ on $[n]$ is disconnected with $\Delta = \Delta_A \oplus \Delta_B$ for a partition $A,B$ of $[n]$, then $\Delta|_A = \Delta_A$ and $\Delta|_B = \Delta_B$. 
			We now provide some characterizations of (dis)connectedness. 
		
		\begin{Proposition}\label{NoetherNum-DirectSum}
			Let $\Delta$ be a matroid on $[n]$, $J=J(\Delta) \sub R$ and $c := \h\,J = r(\Delta)$. The following are equivalent:
			\begin{enumerate}
				\item\label{DirectSum} $\Delta$ is disconnected, i.e. 
				$\Delta = \Delta_A \oplus \Delta_B$;	
				\item\label{NotMinGen}  $(x_1\cdots x_n)t^c$ is not a minimal algebra generator of $\R_s(J)$;
				\item\label{Partition} there exists a partition $A,B$ of $[n]$ such that both $\{ F \cap A : F \in \mathcal{F}(\Delta) \}$ and $\{ F \cap B : F \in \mathcal{F}(\Delta) \}$ forms a set of basis of a matroid;
				
				\item\label{DisjointCoverIdeal} $J = J_A + J_B$ where $J_A$ and $J_B$ are ideals in $k[A]:=k[x_i\,\mid\,i\in A]$ and $k[B]:=k[x_j\,\mid\,i\in B]$ respectively, where $A,B$ is a partition of $[n]$.
			\end{enumerate}
		\end{Proposition}
		
		\begin{proof} 
			$(\ref*{DirectSum}) \Longrightarrow (\ref*{DisjointCoverIdeal})$: By definition of direct sum, the primes in $\Ass(R/J)$ are all of the form $$\Ass(R/J) = \{ \p_{F_1} + \p_{F_2} : F_1 \in \mathcal{F}(\Delta_A), F_2 \in \mathcal{F}(\Delta_B) \}.$$ As $A$ and $B$ are disjoint, for any $F_1 \in \mathcal{F}(\Delta_A)$, we have $$\p_{F_1} + \bigcap_{F_2 \in \mathcal{F}(\Delta_B)} \p_{F_2}= \bigcap_{F_2 \in \mathcal{F}(\Delta_B)} (\p_{F_1} + \p_{F_2}).$$ Hence, it follows that $$\bigcap_{F_1 \in \mathcal{F}(\Delta_A)}\p_{F_1} + \bigcap_{F_2 \in \mathcal{F}(\Delta_B)} \p_{F_2}= \bigcap_{\substack{F_1 \in \mathcal{F}(\Delta_A) \\ F_2 \in \mathcal{F}(\Delta_B)}} (\p_{F_1} + \p_{F_2}) = J$$ Then setting $J_A = J(\Delta_A)$ and $J_B = J(\Delta_B)$, see that $J = J_A + J_B$.

			$(\ref*{DisjointCoverIdeal}) \Longrightarrow (\ref*{NotMinGen})$: Let $a := \h\,J_A$, $b: =\h\,J_B$. As $A$ and $B$ are disjoint, it follows that $c = a + b$. Furthermore, by \cite[Thm~3.4]{HNTT17} or \cite[Thm~7.8]{BocciEtAl16}, $J^{(c)} = \sum_{i+ j = c} J_A^{(i)}J_B^{(j)}$. Finally $x_1\cdots x_n=x_Ax_B$; since $x_A \in J_A^{(a)}$ and $x_B\in J_B^{(b)}$ then $(\ref*{NotMinGen})$ follows. 
			
			$(\ref*{NotMinGen}) \Longrightarrow (\ref*{Partition})$: By symmetry it suffices to show that $\{ F \cap A : F \in \mathcal{F}(\Delta) \}$ satisfies Definition \ref{Def-Matroid-Basis}(1). We first show that $\{ F \cap A : F \in \mathcal{F}(\Delta) \}$ and $\{ F \cap B : F \in \mathcal{F}(\Delta) \}$ are pure simplicial complexes. By assumption, and since $J$ and $\sfp_\ell(J)$ are all monomial ideals, we can write $(x_1\cdots x_n)t^c= (x_At^a)(x_Bt^b)$ for subsets $A,B\subseteq [n]$ and $a,b\in \ZZ_+$ such that $a + b = c$. Since $x_1\cdots x_n$ is squarefree, then $A,B$ forms a partition of $[n]$. Now let $\p=\_F \in \Ass(R/J)$. Since $x_A\in J^{(a)}$, then $|F \cap A| \geq a$. Similarly $|F \cap B| \geq b$. Then 
			$$a+b\leq |F \cap A|+|F \cap B|\leq |F|=c=a+b,$$
			where the rightmost inequality holds because $A,B$ are disjoint. Therefore, all inequalities are equalities and, in particular, $|F \cap A| = a$ and $|F \cap B| = b$. 
			
			We now use purity to prove that Definition \ref{Def-Matroid-Basis}(1) holds for $\{ F \cap A : F \in \mathcal{F}(\Delta) \}$. Let $F,F' \in \mathcal{F}(\Delta)$ with $F' \cap A\neq F \cap A$, and let $i \in (F \cap A) - (F' \cap A)$, so $i \in (F - F') \cap A$. Since $\Delta$ is a matroid, there is a $j \in F' - F$ so that $(F - i) \cup j \in\mathcal{F}(\Delta)$. It remains to show that $j \in A$, so that the exchange of $i$ with $j$ happens in $A$. This holds because, by purity, $| ((F - i) \cup j) \cap A| = a $, and therefore $j \in A$ as $i \in A$. 
				
			$(\ref*{Partition}) \Longrightarrow (\ref*{DirectSum})$: Let $\Delta_A$ and $\Delta_B$ be the simplicial complexes generated by $\{ F \cap A : F \in \mathcal{F}(\Delta) \}$ and $\{ F \cap B : F \in \mathcal{F}(\Delta) \}$, respectively. We prove that $\Delta=\Delta_A\oplus \Delta_B$.
			
			The forward inclusion is immediate, because $A \sqcup B=[n]$. 
			As for the other inclusion, we show $F\cup G$ is in $\mathcal{F}(\Delta)$ for any $F \in \mathcal{F}(\Delta_A)$ and $G \in \mathcal{F}(\Delta_B)$. 
			Since $A,B$ is a partition of $[n]$, by definition of $\Delta_A$ and $\Delta_B$ there exist $F' \in \Delta$ with $F' \subseteq A$, and $G' \in \Delta$ with $G' \subseteq B$, such that $F \cup G' \in \mathcal{F}(\Delta)$ and $F' \cup G \in \mathcal{F}(\Delta)$. We will show that we can ``trade" $G'$ for $G$. By purity of $\Delta$, $|G|=|G'|$. Since $A,B$ is a partition of $[n]$ and $\Delta_A$, $\Delta_B$ are matroids, then there exist $a,b\in \ZZ_+$ such that 
			$$
			(\diamondsuit)\qquad |H\cap A| = a \qquad \text{  and } \qquad |H\cap B| = b\qquad \text{ for any }H \in \mathcal{F}(\Delta).$$
			Set $E' := G' - G \subseteq (F \cup G') - (F' \cup G )$. Since $\Delta$ is a matroid then, by Definition \ref{Def-Matroid-Basis} (3), there is a subset $E \subseteq (F' \cup G) - (F \cup G')$ such that $K:=((F \cup G') - E') \cup E\in \F(\Delta)$. Then we have
			$$
			|G'|=|(F\cup G')\cap B|=|K\cap B|=|((G'-E')\cup E)\cap B|=|(G'-E')\cup (E\cap B)|.
			$$
			The leftmost equality holds because $F\cap B=\emptyset$ and $G'\subseteq B$, the second equality from the left follows from $(\diamondsuit)$, and the rightmost equality holds because $E'\subseteq G'\subseteq B$. Since $|E|=|E'|$, then from the above we obtain $E \subseteq B$ as well. Since $E\sub F'\cup G$ and $F'\sub A$ then, in particular, $E\subseteq G$, so $|G'|=|(G'-E')\cup E|$. Since $(G'-E')\cup E=(G'\cap G)\cup E\subseteq G$, and $|(G'-E')\cup E|=|G'| = |G|$ then $(G' - E') \cup E = G$. Therefore $$
			F \cup G =F \cup ((G' - E') \cup E)=((F \cup G') - E') \cup E = K  \in \F(\Delta).$$\end{proof} Recall that for any $E\subseteq [n]$ we write $x_E:=\prod_{i\in E}x_i$.\begin{Corollary}\label{MinAlgebraGenOfRestriction}  
			Let $\Delta$ be a matroid with cover ideal $J$, let $c = \h\,J$. For any $E\subseteq [n]$, let $r = r(\Delta|_E)$. If $x_Et^{r}$ is not a minimal algebra generator of $\mathcal{R}_s(J)$ then $\Delta|_E$ is disconnected. 
		\end{Corollary}

		\begin{proof} It is easy to see that for any subset $A \sub E$, that if $x_A \in J^{(r)}$ then $x_A \in J(\Delta|_E)^{(r)}$. Hence if $x_At^r$ is not a minimal algebra generator of $\R_s(J)$ then $x_At^r$ is also not a minimal generator of $\R_s(J(\Delta|_E))$. The statement now follows by using $\Delta|_E$ as the matroid in Proposition \ref{NoetherNum-DirectSum} $(\ref*{NotMinGen}) \Longrightarrow (\ref*{DirectSum})$.
		\end{proof}
		
		We can now provide a combinatorial description of the symbolic Noether number.
		\begin{Theorem}\label{NoetherNumber} 
			Let $\Delta$ be a matroid with connected components $\Delta_{A_1},...,\Delta_{A_d}$, where $A_i$ is the ground set of each $\Delta_{A_i}$. Let $J$ be the cover ideal of $\Delta$, then the symbolic Noether number of $\mathcal{R}_s(J)$ is $\max\{ r(\Delta_{A_i})\}$. 
					\end{Theorem}
		
		\begin{proof}
			Let $\Delta_A$ be a connected component of $\Delta$ of maximal rank $r$. Since $\Delta|_A=\Delta_A$ is connected, by Corollary \ref{MinAlgebraGenOfRestriction}, $x_At^r$ is a minimal algebra generator of $\R_s(J)$. Hence the symbolic Noether number of $J$ is at least $r$.
			
			By Proposition \ref{SymReesAlgDescription}, to prove the other inequality it suffices to look at ``squarefree monomial" algebra generators of $\R_s(J)$, i.e. algebra generators of the form $f:=x_Et^k$. To conclude we show that if any such $f=x_Et^k$ is an  algebra generator of $\R_s(J)$ for some $E\subseteq [n]$, with $|E| \geq k > r$, then $f$ is not a {\em minimal} algebra generator of $\R_s(J)$.  For every $i=1,\ldots,d$, let $r_i:=r(\Delta_{A_i})$ and 
			$E_i:=E\cap A_i$. Observe that at least two of the $E_i$'s are non--empty because, if not, then $E\subseteq A_i$ for some $i$, thus $k\leq r_i\leq r$, yielding a contradiction. 
			
			Now, for every $i$, let $m_i:=x_{E_i}\in R$, and let $k_i$ be the ``symbolic degree of $m_i$", i.e. let $k_i:=\max\{h\in \NN_0\,:\,m_i \in J^{(h)}\}$, then $f_i:=m_it^{k_i}$ is an algebra generator of $\R_s(J)$. Note that if $E_i=\emptyset$, then $k_i=0$ and $m_i=f_i=1$. Now, by construction, $f=(f_1f_2\cdots f_d)t^{k-\sum k_i}$, and since at least two of the $f_i$'s are not equal to  $1$, if we show that  $f=f_1\cdots f_d$ then we have concluded the proof. Therefore, we want to prove that $k=\sum_{i=1}^d k_i$.
			Since $k$ is the largest symbolic power of $J$ containing $x_E$ and $f_i\in J^{(k_i)}$ for all $i$, then $k\geq \sum_{i=1}^dk_i$. 
			
			If $\Gamma$ is any matroid, and we write a squarefree monomial $f\in J(\Gamma)^{(k)} - J(\Gamma)^{(k+1)}$ as $f=f_1f_2$, and let $k_i$ be the largest symbolic power of $J(\Gamma)$ containing $f_i$, then in general it may happen (and it frequently happens) that $k>k_1+k_2$. 
			However, we show that, in our setting, equality holds.
			
			By definition of $k_i$, for every $i$ there is some prime $\p_{F_i}\in \Ass(R/J)$ such that $m_i\in \p_{F_i}^{k_i} - \p_{F_i}^{k_i+1}$. Clearly, if $E_i=\emptyset$, then $k_i=0$, so in this case $\p_{F_i}$ can be taken to be {\em any} associated prime. Let $\q_i:=\p_{F_i\cap A_i}$; since $\Delta = \bigoplus_{i=1}^d \Delta_{A_i}$, then $\q_i$ is an associated prime of $R/J(\Delta_{A_i})$. Since $\supp{m_i}\subseteq A_i$, it follows that $m_i\in \q_i^{k_i} - \q_i^{k_i+1}$ for every $i$.  
			Now, let 
			$$
			\p:=\q_1+\q_2+\ldots +\q_d.
			$$ 
			Since $\q_i\in \Ass(R/J(\Delta_{A_i}))$ for every $i$, and $\Delta = \bigoplus_{i=1}^d \Delta_{A_i}$, then $\p\in \Ass(R/J)$. Finally, by definition of $k_i$ and the above, $x_E \in \p^{k_1+\ldots+k_d} - \p^{k_1+\ldots+k_d+1}$. Since $k$ is the smallest integer for which $x_E \in J^{(k)} - J^{(k+1)}$, then $\sum_{i=1}^dk_i\geq k$. 
		\end{proof}
		
		

		\section{Application 1: Formulas for Symbolic defects}
		
		In this section we prove a formula for the symbolic defect of {\em any} \Cmatroidal ideal, a general upper bound for the symbolic defect of {\em any} squarefree monomial ideal, and we show that a large class of matroidal ideals achieve the bound, i.e. they have {\em maximal symbolic defects}. 
		\medskip
		
		We first recall the notion of symbolic defects, as introduced by  Galetto, Geramita, Shin, and Van Tuyl \cite{GGSV}. 
		The {\em $\ell$-th symbolic defect} of an ideal $I$ is $$\sdef(I,\ell):=\mu(I^{(\ell)}/I^{\ell}),$$
		where $\mu(-)$ denotes the minimal number of generators of a graded $R$-module.
		
		By definition, the $\ell$-th symbolic defect is a measure of the difference between $\symp{I}$ and $I^\ell$. 
		In general, these invariants are extremely hard to compute. For instance, it is not even known which squarefree monomial ideals $I$ are {\em normally torsion-free}, i.e. have trivial symbolic defects, $\sdef(I,\ell)=0$ for all $\ell$. A wide--open, long--standing conjecture raised in 1990 by Conforti and Cornuejols, if proved, would characterize normally torsion--free squarefree monomial ideals in terms of the packing problem, see e.g. \cite[Section~4.2]{DDGH}.
		
		As another illustration, let $I_X$ be the defining ideal of a set $X$ of $r$ general points in $\mathbb P^2$. Currently, one only knows the precise values of $\sdef(I_X,2)$ if $r=1,2,3,4,5,7$ or $8$ -- in all these cases it is either 0 or 1 -- while for $r=6$ or $r\geq 9$ the only known result is that $\sdef(I_X,2)>1$, \cite[Thm~6.3]{GGSV}. We are not aware of any results in the literature computing $\sdef(I_X,\ell)$ for any $\ell\geq 3$.  
		
		When $I$ is a monomial ideal, Drabkin and Guerrieri proved that $\sdef(I,\ell)$ is eventually a quasi--polynomial function in $\ell$ \cite[Thm~2.4]{DG}, but little is known about it. In fact, in general, there are very few classes of ideals for which some symbolic defects are known, even when one restricts to squarefree monomial ideals $I$ of height 2 (see e.g. \cite[Sections 4 and 5]{DG}). Some notable exceptions are: 
		\begin{itemize}
			\item complete intersection ideals -- in this case it is well--known that $\sdef(I,\ell)=0$ for all $\ell$;
			\item cover ideals of bipartite graphs -- in this case Herzog, Simis, and Villareal proved $\sdef(I,\ell)$ $=0$ for all $\ell$, \cite[Thm~5.9]{SVV}; 
			\item ideals of star configurations --  an {\em explicit} formula is given in \cite[Cor~4.13]{M}. In this case  $\sdef(I,\ell)>0$ for all $\ell$, and the formula is complicated, further illustrating the difficulty, in general, in the computation of $\sdef(I,\ell)$.
		\end{itemize} 
		We recall that homogeneous polynomials $f_1,\ldots,f_r$ in $R$ form a {\em regular sequence} if every $f_{i}$ is a non--zero divisor on $S/(f_1,\ldots,f_{i-1})$, and an ideal $I$ is a {\em complete intersection} if it can be generated by a regular sequence.\medskip

		To state our first main result of this section we need to recall the notion of monomial grade. First, recall that the {\em grade} of a homogenous ideal $I$, denoted ${\rm grade}(I)$, is the maximal length of a regular sequence in $I$. 
		The {\em monomial grade} of a monomial ideal $I$ is the maximal length ${\rm Mgrade}(I)$ of a {\em monomial }regular sequence (i.e. a regular sequence consisting of monomials) in $I$. If $I$ is \Cmatroid then writing $I=I_\Delta$ for some matroid $\Delta$, then ${\rm Mgrade}(I)$ is the largest size of an independent set in the circuit graph $G(\Delta)$ defined in Definition \ref{Def-CircuitGraph}.  
		
		By definition, ${\rm Mgrade}(I)\leq {\rm grade}(I)$, and the ideal $I$ is called {\em K\"onig} if equality holds. In general, the above inequality is strict. For instance, it is proved in \cite{LM} that a matroidal ideal is K\"onig if and only if $I$ is a complete intersection. In fact, matroidal ideals tend to have relatively small monomial grades. E.g. if $\Delta$ is a paving matroid (see Section \ref{Section-Paving}) of rank $n-c$, then ${\rm Mgrade}(I_\Delta)\leq \min\left\{c,\, \frac{n}{n-c}\right\}$.
		
		In general, if $I$ is a squarefree monomial ideal and $\alpha(I)$ denotes the smallest degree of an element of $I$, then it is easily seen that ${\rm Mgrade}(I)\leq \min\left\{{\rm grade}(I),\, \frac{n}{\alpha(I)}\right\}$. 
		
		We  now prove a formula for the symbolic defects of any \Cmatroid ideal.

		\begin{Theorem}\label{Thm-Sdef}
			Let $I$ be a \Cmatroid ideal. 
			Then for any $\ell\geq 1$ 
			$$\sdef(I,\ell) = \mu(\symp{I}) - \sum_{r=1}^{{\rm Mgrade}(I)} |\sfp(I^r)\cap G(I^{(r)})| {\ell-1 \choose \ell - r}.$$

		\end{Theorem}

		\begin{proof}
			Recall that $\m$ denotes the ideal of the variables of $R$. Let $A_r:=\sfp(I^r)\cap G(I^{(r)})$.
			First, we note that for any monomial $m\in A_r$, up to reordering, there is only {\em one} way to write $m$ in the form $m=m_{i_1}\cdots m_{i_r}$ for some $i_1,\ldots,i_r \in [s]$. To see this,  assume by contradiction one can also write $m$ as $m=m_{j_1}\cdots m_{j_r}$ with $m_{j_1}\notin \{m_{i_1},\ldots,m_{i_r}\}$. Then, in particular, there is a variable $x\in \supp{m_{i_1}}-\supp{m_{j_1}}$. Since all the $m_{i_h}$ have disjoint support, it follows that $x$ does not divide $\LCM(m_{j_1},m_{i_2},\ldots,m_{i_r})$ and since $m_{j_1}$ divides $m$ it follows that $\LCM(m_{j_1},m_{i_2},\ldots,m_{i_r})$ properly divides $\LCM(m_{j_1},m_{i_2},\ldots,m_{i_r})=m_{i_1}\cdots m_{i_r}=m$. Since $\LCM(m_{j_1},m_{i_2},\ldots,m_{i_r})\in \symp{I}$ by Proposition \ref{Matroid-LCM}, then $m\in G(\symp{I})$, yielding a contradiction. 
			
			Note that, for any subset $T\subseteq \{i_1,\ldots,i_r\}$, the monomial $m_T:=\prod_{j\in T}m_j$ is in $A_{t}$ where $t := |T|$. Indeed, if not, then $m_T\in \m I^{(t)}$, so $m=m_T(m/m_T)\in \m I^{(t)}I^{r-t}\subseteq \m I^{(r)}$ contradicting the assumption $m\in G(I^{(r)})$.

			To prove the equality in the statement, we prove the following 
			qualitative statement: 
			{\small $$(\#)\quad
				I^\ell \cap G(\symp{I}) = \{m_{i_1}^{b_1}\cdots m_{i_r}^{b_r}\,\mid\, m_{i_1}\cdots m_{i_r}\in A_r \textrm{ for some } r \geq 1, \,\sum b_h=\ell,\, b_h\geq 1\,\forall\,h\}.
				$$}
			
			``$\supseteq $" Let $m_{i_1}\cdots m_{i_r}\in A_r$ for some $r\geq 1$. 	
			Now, let $M=m_{i_1}^{b_1}\cdots m_{i_r}^{b_r}$ for any $b_h\geq 1$ with $\sum_{h=1}^r b_h=\ell$.  The sum condition implies that $M\in I^\ell$. Set $b = \max\{b_1,...,b_r\}$, and for $1 \leq j \leq b$, let $T_j = \{ i_h :  b_h \geq j\}$. Notice that $M = m_{T_1}\cdots m_{T_b}$, and $\sum_{j=1}^{b}|T_j| = \sum_{h=1}^{r}b_h = \ell$. It follows from the previous paragraph that $m_{T_j}\in G(I^{(|T_j|)})$ for all $j$. By construction $\supp{m_{T_j}} \supseteq \supp{m_{T_{j+1}}}$, hence, by the Structure Theorem \ref{MatroidSymPowerThm}, $M \in G(\symp{I})$. 
			
			``$\subseteq$" 
			Let $M\in I^\ell \cap G(\symp{I})$. Since $M\in I^\ell$, then there exist $m_{i_h}\in G(I)$ and $b_h\geq 1$ with $\sum b_h=\ell$ such that $M=m_{i_1}^{b_1}\cdots m_{i_r}^{b_r}$. Observe that, since $M\in G(\symp{I})$, then  $m:=m_{i_1}\cdots m_{i_r}\in G(I^{(r)})$,  by Remark \ref{Rmk-Minimal}. 
			Additionally, since $M\in G(\symp{I})$, then the $m_{i_h}$'s have pairwise disjoint supports, by Corollary \ref{Cor-Disjoint}.   
			Therefore, $m\in A_r$. This concludes the proof of $(\#)$. 
			\bigskip

			Since $\sdef(I,\ell)=\mu(I^{(\ell)})- |I^{\ell}\cap G(\symp{I})|$ to complete the proof we determine the cardinality of the right--hand side of $(\#)$. 
			We fix an arbitrary total order on $G(I)$ and from now on we assume that any monomial $m=m_{i_1}\cdots m_{i_r}\in A_r$ is written so that $m_{i_1}<m_{i_2}<\ldots<m_{i_r}$.	 
			From the first paragraph of this proof, after fixing the order on $G(I)$, there is a unique way to write any monomial $M\in I^\ell \cap G(\symp{I})$ as $M=m_{i_1}^{b_1}\cdots m_{i_r}^{b_r}$ with $m_{i_1}<m_{i_2}< \ldots < m_{i_r}$, $b_h\in \ZZ_+$ for all $h$, and $\sum_{h=1}^r b_h=\ell$. In other words, $M$ is uniquely determined by  $\sqrt{M}=m_{i_1}\cdots m_{i_r}$ and the exponent vector $(b_{1},\ldots,b_r)$. Therefore	
			$$I^\ell \cap G(\symp{I}) = \bigsqcup_{r=1}^{s} D_{r,\ell},$$
			where $D_{r,\ell}:=\{M\in I^\ell \cap G(\symp{I})\,\mid\, \sqrt{M}\in A_r\}$.

			To obtain the desired formula from the disjoint decomposition above, we need to determine $|D_{r,\ell}|$. By the above, each $M\in D_{r,\ell}$ is uniquely determined by  $\sqrt{M}=m_{i_1}\cdots m_{i_r}$ $\in A_{r}$ and the exponent vector $(b_1,\ldots,b_r)$. Therefore, for any $m\in A_r$ there are as many $M\in D_{r,\ell}$  with $\sqrt{M}=m$  as monomial expression in $\{m_{i_1},\ldots,m_{i_r}\}$  with exponents $b_h \geq 1$ with $\sum b_h=\ell$. Since $b_h\geq 1$ for all $h$, this is the same as the number of monomial expression in $r$-variables of degree $\ell - r$, which is ${r + (\ell -r) - 1 \choose \ell - r} = {\ell - 1 \choose \ell - r} $. 
			It follows that 
			$$(\spadesuit) \qquad \sdef(I, \ell) = \mu(\symp{I}) - \sum_{r = 1}^{s} |D_{r,\ell}|{\ell-1 \choose \ell-r}.$$
			
			Finally, to show that the sum runs from $r=1$ to $r={\rm Mgrade}(I)$, we show that $|D_{r,\ell}|=0$ for all $r>{\rm Mgrade}(I)$. To see it, assume $m=m_{i_1}\cdots m_{i_r} \in A_r$ for some $r\geq 1$. Since $m$ is squarefree, then the supports of the $m_{i_h}$'s are pair-wise disjoint, therefore $m_{i_1},\ldots,m_{i_r}$ is a monomial regular sequence in $I$, so $r\leq {\rm Mgrade}(I)$. This proves that $|A_r|=0$ for all $r>{\rm Mgrade}(I)$, concluding the proof.\end{proof}

		Philosophically, by Theorem \ref{Thm-Sdef}, the $\ell$-th symbolic defects of \Cmatroid ideals are uniquely determined by the {\em squarefree} elements in $I^\ell \cap G(\symp{I})$. As the proof heavily relies on the Structure Theorem \ref{MatroidSymPowerThm}, it is not surprising that an analogous principle, for non--\Cmatroid ideals, is in general false. E.g. $I = (ab,ac,bcd) \sub k[a,b,c,d]$, $\sdef(I,2) = 1$ but the formula produces $2$. 
		
		Furthermore, it is not true that a squarefree monomial ideal $I$ is \Cmatroid if and only if the same formula of Theorem \ref{Thm-Sdef} holds. E.g. $I=(abc, abd, acd, bcde) \sub k[a,b,c,d,e]$ is not \Cmatroid but its symbolic defect does satisfy the formula. 
		
		For future references, we compute the contribution of $r=1$ to the sum in Theorem \ref{Thm-Sdef}:
		\begin{Remark}\label{Rmk-a_1}
			Let $I$ be \Cmatroid and $D_{r,\ell}:=\{M\in I^\ell \cap G(\symp{I})\,\mid\, \sqrt{M}\in \sfp(I^r)\cap G(I^{(r)})\}$. Then
			$$|\sfp(I)\cap G(I)|=|D_{1,\ell}|{\ell -1 \choose \ell -1} =  \mu(I).$$ 
		\end{Remark}
		
		By Theorem \ref{Thm-Sdef}, for a fixed \Cmatroid ideal $I$, all symbolic defects are  determined by $$a_r:=|\sfp(I^r)\cap G(I^{(r)})|.$$
		The good news is that, for a fixed \Cmatroid ideal, it is usually not hard to compute the $a_r$'s.
		
		\begin{Example} Consider the $I = (ab, ace,ade,aef,bce,cd,cf,bde,bef,df)$ which is \Cmatroid (e.g. by Theorem $\ref{Thm-Matroid-Equiv}$ (5)), and has ${\rm Mgrade}(I) = 2$. We explicitly write the sets $\sfp(I^r)\cap G(I^{(r)})$ for $1\leq r \leq 2$. Of course $\sfp(I) \cap G(I) = G(I)$, hence $a_1 = 10$. Also, $$\sfp(I^2) \cap G(I^{(2)}) = \{ (ab)(cd), (ab)(cf), (ab)(df) \} \qquad \Lra \qquad a_2 = 3.$$
			Now, as an example, for $\ell = 2$, one can verify that $\mu(I^{(2)}) = 23$, hence the formula in Theorem \ref{Thm-Sdef} produces $\sdef(I,2) = 10$.
		\end{Example}
		
		Another reason why the numbers $a_r$ above are often not hard to compute is that, 
		for large classes of \Cmatroid ideals, all the $a_r$'s vanish, except $a_1=\mu(I)$. (see Theorem \ref{Thm-Sdef2} and Remarks \ref{Rmk-sdef2} and \ref{Rmk-Conj}.)
		\medskip
		We now take a small detour to prove that the same formula applies to the ideal of any matroidal configuration.
		In \cite[Thm~3.6(2)]{GHMN} the authors prove that for a \Cmatroid ideal $I$, a minimal graded free resolution of $\symp{(I_*)}$ is obtained by ``specializing" a minimal graded free resolution of $\symp{I}$, and, as a consequence of their result, a number of numerical invariants of $\symp{(I_*)}$ are obtained from the corresponding invariants of $\symp{I}$ (e.g. \cite[Cor.~4.3 and 4.6, and Thm~4.8]{GHMN}). However, it is not clear whether similar formulas hold for numerical invariants obtained, for instance, by looking at quotients, or inclusions of some ideals in some other ideals, see e.g. \cite[Question~4.7]{GHMN} or the fact that \cite[Prop~3.8(1)]{GHMN} is not stated as an ``if and only if" statement in full generality. (see also \cite[Rmk~3.2]{M}.) 
		
		Next, we prove that the symbolic defect, despite being obtained by computing minimal generators of a quotient, is actually preserved by specializations: 
		
		\begin{Theorem}\label{Thm-Sdef-Spec}
			Let $I$ be a \Cmatroid ideal and let $I_*$ be a specialization of $I$. Then for all $\ell \geq 1$,
			$$\sdef(I_*,\ell)=\sdef(I,\ell) = \mu(\symp{I}) - \sum_{r=1}^{{\rm Mgrade}(I)} |\sfp(I^r)\cap G(I^{(r)})| {\ell-1 \choose \ell - r}.$$
		\end{Theorem}
		
		\begin{proof}
			By Theorem \ref{Thm-Sdef} we only need to prove the first equality. 
			Write $G(I)=\{m_1,...,m_s\}$ and let $G_\ell:=G(\symp{I})-G(I^\ell)$.  Now, the statement follows from the stronger, qualitative statement that $\varphi(G_\ell)$ is a {\em minimal} generating set of $\symp{(I_*)}/(I_*)^\ell$. 
			
			Since, by \cite[Thm~3.6(1)]{GHMN}, $\varphi(G(\symp{I}))$ is a generating set of $\symp{(I_*)}$, and since $\varphi(G(I^\ell))$ is a generating set of $(I_*)^\ell$, then $\varphi(G_\ell)$ is a generating set of  $\symp{(I_*)}/(I_*)^\ell$. 
			
			We prove minimality. First, since $I^\ell\subseteq \symp{I}$, then $I^\ell \cap G(\symp{I})=G(I^\ell)\cap G(\symp{I})$. Let $\{N_1,\ldots,N_u\}=I^\ell \cap G(\symp{I})$, and observe that, if $N\in G(I^\ell)$ and $N\notin \{N_1,\ldots,N_u\}$, then $N\in \m \symp{I}$, so $\varphi(N)\in \varphi(\m)\varphi(\symp{I})=(f_1,\ldots,f_n)\symp{(I_*)}$. 
			
			Next, assume by contradiction there exists $M\in G_\ell$ such that $F:=\varphi(M)$ is not a minimal generator of $\symp{(I_*)}/(I_*)^\ell$. 
			We observe that, since $M\not\in G(I^\ell)$, then also $F\notin\varphi(G(I^\ell))$, because otherwise, $\varphi(M)=\varphi(m_{i_1}\cdots m_{i_\ell})$ for some $i_1,\ldots,i_\ell\in [s]$. Then \cite[proof of Prop~3.8]{M} shows that $M=m_{i_1}\cdots m_{i_\ell}$, 
			thus $M\in I^\ell$, 
			yielding a contradiction. 
			
			Therefore, $\varphi(I^\ell \cap G(\symp{I}))=\{\varphi(N_1),\ldots,\varphi(N_u)\}\subseteq \varphi(G(\symp{I}) -\{M\})$, so $\varphi(N_i)\in J_*$ for all $i=1,\ldots,u$, where $J$ is the $R$-ideal generated by $G(\symp{I}) -\{M\}$. 
			
			Now, the non--minimality of $F$ is equivalent to the inclusion $F\in J_* + (I_*)^\ell$. We claim that, after possibly including in $J_*$ any element of $(I_*)^\ell $ which is a minimal generator of $\symp{(I_*)}$, we may further assume $F\in J_* + (f_1,\ldots,f_n)\symp{(I_*)}$.
			
			More in details, write $F=F_0 + \sum e_{i_1,\ldots,i_\ell}F_{i_1}\cdots F_{i_\ell}$ for some $F_0\in J_*$, $1\leq i_h\leq s$, $e_{i_1,\ldots,i_\ell}\in S$ and monomials $F_{i_h}$ in $f_1,\ldots,f_n$. For each $i_1,\ldots,i_\ell$, if $m_{i_1}\cdots m_{i_\ell}\in I^\ell\cap G(\symp{I})=\{N_1,\ldots,N_u\}$, then $F_{i_1}\cdots F_{i_\ell}\in (\varphi(N_1),\ldots,\varphi(N_u))\subseteq J_*$, so we may include these terms in $F_0$ to assume $F=F_0+ \sum e_{i_1,\ldots,i_\ell}F_{i_1}\cdots F_{i_\ell}$ as above, with the additional assumption that the sum runs over all indices $i_1,\ldots,i_\ell\in [s]$ such that $N:=m_{i_1}\cdots m_{i_\ell} \notin G(\symp{I})$. We proved above that for any such $N$ the image $\varphi(N)=F_{i_1}\cdots F_{i_\ell}$ is in $(f_1,\ldots,f_n)\symp{(I_*)}$, so 
			$F\in J_*  + (f_1,\ldots,f_n)\symp{(I_*)}$, proving the desired claim. 
			Therefore, 
			$$\symp{(I_*)}=(F)+J_* \subseteq J_*  + (f_1,\ldots,f_n)\symp{(I_*)}.$$
			By Nakayama's Lemma it follows that $\symp{(I_*)}=J_*$, i.e. $F$ is not part of a minimal generating set of $\symp{(I_*)}$, which contradicts \cite[Thm~3.6(1)]{GHMN}.\end{proof}

		Next, we want to provide a wide class of \Cmatroid ideals for which we can provide an explicit formula for $\sdef(I,\ell)$.  
		It follows by Theorem \ref{Thm-Sdef} and Remark \ref{Rmk-a_1} that for any \Cmatroid ideal,
		$$
		\sdef(I,\ell)\leq \mu(\symp{I}) - \mu(I).
		$$ 
		We next observe that the same bound holds for any squarefree monomial ideal $I$. Recall that if $G(I)=\{m_1,\ldots,m_s\}$, then $G(I^{[\ell]})=\{m_1^\ell,\ldots,m_s^{\ell}\}$.
		
		\begin{Proposition}\label{Prop-Sdefl}
			For any squarefree monomial ideal $I$
			one has
			$$
			\sdef(I,\ell)\leq \mu(\symp{I}) - \mu(I),
			$$
			and equality holds if and only if $G(\symp{I})\cap G(I^\ell)=G(I^{[\ell]}).$\\
			Moreover, the above bound is sharp. In fact, for any integers $n>c\geq 2$ there exists a matroidal ideal $I$ of $\h\,I=c$ in $R=k[x_1,\ldots,x_n]$ for which ``$=$" holds.
		\end{Proposition}
		
		\begin{proof}
			Let $\Delta$ be a simplicial complex with $I=J(\Delta)$. Viewing elements of $G(\symp{I})$ as basic $\ell$-covers, we see that every element of the monomial bracket power $I^{[\ell]}$ will always be a basic $\ell$-cover. Hence we always have $G(I^{[\ell]}) \sub G(\symp{I}) \cap G(I^\ell)$, therefore 
			$$\sdef(I,\ell) \leq \mu(\symp{I}) - \mu(I^{[\ell]}) = \mu(\symp{I}) - \mu(I).$$
			Next, for any $n>c\geq 2$, let $U_{c,n}$ be the uniform matroid of rank $c$ on $V=[n]$ and let $I=J(U_{c,n})$. By \cite[Thm~4.11]{M}, for all $\ell\geq 1$ one has $\sdef(I,\ell)= \mu(I^{(\ell)})-\mu(I)$. 
		\end{proof}
		
		\begin{Remark}
			The above inequality does not hold for general ideals, even if we restrict to Cohen--Macaulay, radical ideals, as the following example shows.
		\end{Remark}
		
		\begin{Example}
			Let $X=\{P_1,\ldots,P_6\}$ be a set of 6 general points in $\mathbb P^2$, then $I_X$ has 4 minimal generators, i.e. $\mu(I_X)=4$, and it can be proved that $\mu(I_X^{(2)})=4$ too. To see this, one first observes that, by generality, one may  assume that for any $j=1,2,3$, the only quadric $q_j$ passing through all the $P_i$'s except $P_j$ is irreducible. For any $1\leq i < j \leq 6$, let $L_{ij}$ be the line determined by $P_i$ and $P_j$. It is not hard to show that 
			$$I_X^{(2)}=(q_1q_2L_{12},\; q_1q_3L_{13},\; q_2q_3L_{23},\; q_1q_2q_3).$$ 
			
			Now, if it were true that the inequality $\sdef(I_X,2)\leq \mu(I_X^{(2)}) - \mu(I_X)$ holds, then $\sdef(I_X,2)$ $=0$, which would contradict the inequality $\sdef(I_X,2)>1$ proved in \cite[Thm~6.3]{GGSV}. 
			
			In fact, the above description of the minimal generators of $I_X^{(2)}$ implies, for degree reasons, that $\sdef(I_X,2)\geq 3$. It is not hard to show that, actually,  $\sdef(I_X,2)=3$. 
		\end{Example}

		As the upper bound in Proposition \ref{Prop-Sdefl} is sharp, we can now define ideals with maximal symbolic defects.
		\begin{Definition}\label{Def-MaxSymDef}
			We say that a squarefree monomial ideal $I$ has {\em maximal $\ell$-th symbolic defect} if 
			$\sdef(I,\ell) = \mu(\symp{I}) - \mu(I)$.
		\end{Definition}
		
		As stated in Proposition \ref{Prop-Sdefl}, there are \Cmatroid ideals with maximal symbolic defects. We note that there exist non--\Cmatroid ideals with maximal symbolic defects as well, e.g. $I=(abc, abd, acd, bcde)\sub k[a,b,c,d,e]$ has maximal symbolic defects for every $\ell$ and it is the Stanley--Reisner ideal of a pure, non--matroidal simplicial complex. For a non--pure example, one can take $I=(abd, acd, bcd)\sub k[a,b,c,d]$. 
		
		\begin{Remark}\label{Rmk-sdef2}
			Let $I$ be \Cmatroid and $\ell\geq 2$, it follows from Theorem \ref{Thm-Sdef} and Remark \ref{Rmk-a_1} that $I$ has  {\em maximal $\ell$-th symbolic defect} if and only if $I^r\cap G(I^{(r)})=\emptyset$ for all $r\geq 2$. 
		\end{Remark}
		
		One may think that having maximal symbolic defects is a very strong condition. In contrast to it, we first characterize  the \Cmatroid ideals with maximal symbolic defect and then we prove that many (conjecturally: almost all) \Cmatroid ideals have maximal symbolic defect! (See also Remark \ref{Rmk-Conj}.) To this end, we need to introduce a couple of combinatorial definitions. 
		\begin{Definition}\label{Def-CircuitGraph}
			Let $\Delta$ be a matroid. 
			
			(1) The {\em (vertex--labeled) circuit graph} of $\Delta$ is a graph $G(\Delta)$ whose vertices are labeled by the circuits of $\Delta$, and there is an edge between two vertices if and only if the corresponding circuits are not disjoint.
			
			(2) We say that $G(\Delta)$ is {\em 2--locally connected} if for any two non--adjacent vertices $C_1,C_2\in G(\Delta)$ there exists a vertex $\widetilde{C}\notin \{C_1,C_2\}$ such that $C_1,\widetilde{C},C_2$ is a path in $G(\Delta)$ and $\widetilde{C} \subseteq C_1\cup C_2$. We call $C_1,\widetilde{C},C_3$ a {\em locally connected 2--path} between $C_1$ and $C_2$.
		\end{Definition}
		
		The motivation in studying (1) -- which is a variation of the well--known concept of circuit graph of a matroid --  comes from the observation that $G(\Delta)$ is connected if and only if $\Delta$ is a connected matroid. 
		Condition (2) is inspired by the notion of ``local connectedness" of a Hochster--Huneke graph \cite[Def~2.11]{Ho}. It may appear strong, but one should compare it with the following observation: If $\Delta$ is a connected matroid with at least 3 circuits, then $G(\Delta)$ has diameter at most 2. (and the diameter is precisely 2 if and only if $\Delta$ has two disjoint circuits.)
		Therefore, given any two non--adjacent vertices in $G(\Delta)$, there is always at least a 2--path (i.e. a path of length 2) connecting them. Condition (2) asks for the existence of a {\em locally connected} 2--path.
		
		To provide further evidence, we now show that a wide class of matroids have (vertex--labelled) circuit graphs which are  2--locally connected. 
		
		\begin{Proposition}\label{Prop-2-Loc}
			Let $\Delta$ be a matroid, let $d(\Delta):=\min\{|C_1|+|C_2|\,\mid\,C_1,C_2$ are disjoint circuits$\}$. If $d(\Delta)\geq r(\Delta)+3$, then $G(\Delta)$ is 2-locally connected.
		\end{Proposition}
		In particular, every matroid whose smallest circuit has size at least $(r(\Delta)+3)/2$ sastisfies the assumptions of Proposition \ref{Prop-2-Loc}. This includes, for instance, all paving matroids (see Section \ref{Section-Paving}) having rank $>2$. 
		
		\begin{proof}
			Let $r:=r(\Delta)$ and $d:=d(\Delta)$. Since every circuit has size at most $r+1$, then it follows by the assumption that every circuit has size at least 2. 
			Let $C_1,C_2$ be any two disjoint circuits and fix $v_i,w_i\in C_i$ for $i=1,2$. Let $D_i:=C_i - v_i$ for $i=1,2$, by assumption $D_1\cup D_2$ has size at least $r+1$. Let $D$ be any subset of $D_1\cup D_2$ of size $r+1$ containing $w_1,w_2$. Observe that since $v_1,v_2\notin D$, then $C_i\not\subseteq D$ for $i=1,2$.
			
			Since $|D|>r$, then $D$ is a dependent set, so it contains a circuit $\widetilde{C}$ which, by the above, is distinct from $C_1,C_2$ and, by construction, $\widetilde{C}\subseteq D\subseteq C_1\cup C_2$.  Therefore, $C_1,\widetilde{C},C_2$ is a locally connected 2--path in $G(\Delta)$.\end{proof}

		We now return to the discussion on the symbolic defects. We have seen that
		for matroidal ideals associated to a uniform matroid one has $\sdef(I,\ell)= \mu(I^{(\ell)})-\mu(I)$ for all $\ell\geq 1$. The Structure Theorem \ref{MatroidSymPowerThm} allows us to extend this equality to a wide class of matroidal ideals, which we characterize combinatorially. In addition, we prove the {\em a priori} unexpected fact that, if the equality $\sdef(I,\ell)= \mu(I^{(\ell)})-\mu(I)$ holds {\em for some} $\ell\geq 2$, then it holds {\em for all} $\ell\geq 2$.

		\begin{Theorem}\label{Thm-Sdef2} Let $I$ be a \Cmatroid ideal, and let $\Delta$ be a matroid such that $I=I_\Delta$.  The following are equivalent:
			
			\begin{enumerate}
				\item\label{connectedcircuits} $G(\Delta)$ is 2--locally connected,
				\item\label{sdefl} $I$ has maximal $\ell$-th symbolic defect for all $\ell$,
				\item\label{sdefl2} $I$ has maximal $\ell$-th symbolic defect for some $\ell\geq 2$,
				\item\label{sdef2} $\sdef(I,2) = \mu(I^{(2)}) - \mu(I)$,
				\item\label{connectedcircuits2} For any $m_1,m_2 \in G(I)$ with disjoint support, there is $\widetilde{m} \notin \{m_1,m_2\}$ such that $\widetilde{m} \, |\, m_1m_2$.
			\end{enumerate}
			
		\end{Theorem}
		\begin{proof} 
			$(\ref*{connectedcircuits}) \Llra (\ref*{connectedcircuits2})$ follows from Proposition \ref{Basic-Matroid-Properties} (2).\\
			$(\ref*{connectedcircuits2}) \Longrightarrow (\ref*{sdefl})$. 
			By Proposition \ref{Prop-Sdefl}, we may assume by contradiction that for some $\ell$, $\sdef(I,\ell) < \mu(\symp{I}) - \mu(I)$. Then there exists $m \in G(\symp{I})$ such that $m \in G(I^\ell) - G(I^{[\ell]})$. We write $m = m_1\cdots m_\ell$ for $m_i\in G(I)$.  As $m \notin I^{[\ell]}$, we may assume $m_1\neq m_2$. 
			
			Since $m\in G(\symp{I})$ the, by Corollary \ref{Cor-Disjoint}, the $m_i$'s have pairwise disjoint support.
			
			Now, by property (\ref*{connectedcircuits2}) there exists $\widetilde{m} \in G(I)$ with $\widetilde{m} \notin \{m_1,m_2\}$ such that $\widetilde{m} \, |\, m_1m_2$. Since $\widetilde{m}\neq m_1$, there exists a variable $y\in \supp{m_1} - \supp{\widetilde{m}}$. Note that $y\notin \supp{m_2}$ by the above. Let $L:=\LCM(\widetilde{m},m_2)$. Then, by the above, $L\,|\,m_1m_2$ and $y\notin \supp{L}$; additionally, by Proposition \ref{Matroid-LCM},  $L\in I^{(2)}$. Since $m_1m_2/y$ is a multiple of $L$, then $m_1m_2/y\in I^{(2)}$. Therefore,  $m/y \in \symp{I}$, contradicting the assumption $m\in G(\symp{I})$.

            $(\ref*{sdefl2}) \Longrightarrow (\ref*{sdefl})$. By Theorem \ref{Thm-Sdef} and Remark \ref{Rmk-a_1}, for any $\ell \geq 2$, $I$ has maximal $\ell$-th symbolic defect for some $\ell\geq 2$ if and only if $|\sfp(I^r) \cap G(I^{(r)})| = 0$ for all $r > 1$. Hence, again, by Theorem \ref{Thm-Sdef}, $|\sfp(I^r) \cap G(I^{(r)})| = 0$ for all $r > 1$ implies that $I$ has maximal $\ell$-th symbolic defect for any $\ell$.
			
			$(\ref*{sdefl}) \Longrightarrow (\ref*{sdef2})$ is clear. $(\ref*{sdef2}) \Longrightarrow (\ref*{connectedcircuits2})$ 
			Let $m_1,m_2\in G(I)$ have disjoint support, then $m_1m_2$ is a squarefree monomial in $I^2\subseteq I^{(2)}$. By Proposition \ref{Prop-Sdefl}, the assumption implies that $m_1m_2\notin G(I^{(2)})$, so there exists $m_0 \in G(I^{(2)})$ such that $m_0$ properly divides $m_1m_2$. Necessarily $m_0$ is squarefree and, by Corollary \ref{Corr-Matroid-LCM}, $m_0 = \LCM(m_3,m_4)$ for $m_3,m_4\in G(I)$. Since $m_0=\LCM(m_3,m_4)$ properly divides $m_1m_2=\LCM(m_1,m_2)$, then $\{m_3,m_4\}\neq \{m_1,m_2\}$, thus one of $m_3$ or $m_4$ is the required $\widetilde{m}$.
		\end{proof}
		
		By Remark \ref{Rmk-sdef2}, $I^r\cap G(I^{(r)})=\emptyset$ for all $r\geq 2$ $\Llra$ $G(\Delta)$ is 2--locally connected. Similarly, for any $q\geq 2$ one can characterize  the last value of $r$ providing a non--zero contribution to the sum in the formula of Theorem \ref{Thm-Sdef}. 
		
		For any \Cmatroid ideal $I=I_\Delta$ and any $q\geq 2$ we can say that $G(\Delta)$ is {\em $q$--locally star connected}, if for any $q$ non-adjacent vertices in $G(\Delta)$, there is a vertex $v$ of $G(\Delta)$ which is adjacent to all of them. (equivalently, the induced graph on these $q+1$ vertices is a star having $v$ as a center.) Note that being 2--locally connected is the same as being 2--locally {\em star} connected.
		Then, the same arguments as above prove the following generalization:
		\begin{Theorem}
			Let $I=I_\Delta$ be a \Cmatroid ideal and $q\geq 2$. Then $G(\Delta)$ is {\em $q$--locally star connected} 
			if and only if $\sdef(I,\ell)=\mu(\symp{I}) - \sum_{r=1}^{\rm q} |\sfp(I^r)\cap G(I^{(r)})| {\ell-1 \choose \ell - r}$ for all $\ell\geq 1.$
		\end{Theorem}

		\section{Applications 2: initial degrees and resurgence}
		
		To illustrate another potential use of Theorem \ref{MatroidSymPowerThm}, we now provide  explicit formulas for the initial degrees of {\em any} symbolic power $I^{(\ell)}$ of any \Cmatroid ideal $I$, as well as for the Waldschmidt constant of $I$. 
		
		In order to prove these results, we need to establish some notation and recall a couple of facts. For a homogeneous ideal $I$ in a polynomial ring $R$, one defines $$\alpha(I) = \min \{ \deg f : f \in I\}.$$
		The {\em Waldschmidt constant }of $I$ is defined to be $$\wdh{\alpha}(I) = \lim_{\ell \to \infty} \frac{\alpha(\symp{I})}{\ell}.$$
		One can show that $\alpha(\symp{I})$ is subadditive so $\lim \frac{\alpha(\symp{I})}{\ell} = \inf \frac{\alpha(\symp{I})}{\ell}$ exists \cite[Lem 2.3.1]{BH}. 
		
		As a direct consequence of the structure theorem, we will show for matroid ideals, that there is a natural description of the initial degrees of the symbolic powers, and the Waldschmidt constant is obtained by a squarefree generator of some symbolic power of that ideal. We first need an elementary inequality.
		
		\begin{Lemma}\label{Lem-Ineq} Let $d_1$,$d_2$, $s_1$, and $s_2$ be positive integers. Then $$\frac{d_1+d_2}{s_1+s_2} \geq \min\left\{\frac{d_1}{s_1}, \frac{d_2}{s_2}\right\}.$$
		\end{Lemma}
		
		\begin{proof} Without lost of generality we may assume $\frac{d_1}{s_1} \leq \frac{d_2}{s_2}$. The proof is a string of arithmetic manipulations. From the assumption $s_2d_1 \leq s_1d_2$. Adding $d_1s_1$ to both sides we obtain 
			$$s_2d_1 + d_1s_1 \;\leq\; d_2s_1 + d_1s_1
			\Lra d_1(s_1+s_2) \;\leq\; s_1(d_2+d_1)
			\Lra \frac{d_1}{s_1} \;\leq\; \frac{d_1 + d_2}{s_1+s_2}.$$
		\end{proof}
		
		The second equality in the next result is a slightly more refined version of the equality proved in \cite[Thm~3.6]{DG}. Although we are sure it is known to experts, we include a short proof for the sake of completeness.
		
		\begin{Theorem}\label{Wadlschmidt-Fin-Gen-Sym-Alg} 
			Let $I \sub R$ be a homogeneous ideal. Assume the symbolic Rees algebra $\R_s(I)$ is Noetherian, and write  $\R_s(I) = R[I_1t,I_2t^2,...,I_ct^c]$ for some $c\geq 1$ and ideals $I_\ell \sub I^{(\ell)}$. Then 
			$$\alpha(I^{(\ell)}) = \min\left\{\sum_{h=1}^c a_h\alpha(I_h)\,\mid\,\sum_{h=1}^cha_h = \ell\right\},$$
			and $$\wdh{\alpha}(I) = \min_{1 \leq \ell \leq c}\left\{\frac{\alpha(I_\ell)}{\ell}\right\}.$$
		\end{Theorem}
		
		\begin{proof} 
			The assumption on $\R_s(I)$ implies that 
			{\small 	$$I^{(\ell)} = \sum_{a_1+2a_2\ldots+ca_c=\ell,\,a_h\geq 0} I^{a_1}I_2^{a_2}\cdots I_c^{a_c} \;\;\Lra\;\;\alpha(I^{(\ell)})=\min\left\{\sum_{h=1}^c a_h\alpha(I_h)\,\mid\,\sum_{h=1}^cha_h = \ell\right\}.$$}
			
			We prove the second equality. Let $a = \min\{ \ell : \frac{\alpha(I^{(\ell)})}{\ell} = \wdh{\alpha}(I) \}$. Now let $m \in I^{(a)}$ with $\deg(m)=a\wdh{\alpha}(I)$. Suppose $m\notin I_{a}$, then, by the description of $\R_s(I)$ we can write $m = m_1m_2$ for some $m_i \in I^{(a_i)}$ with $1\leq a_i<a$ for $i=1,2$ and $a_1 + a_2 = a$. Then, by Lemma \ref{Lem-Ineq},
			$$\wdh{\alpha}(I) = \frac{\alpha(I^{(a)})}{a} \geq \min \left\{\frac{\alpha(I^{(a_1)})}{a_1}, \frac{\alpha(I^{(a_2)})}{a_2}\right\}.$$ 
			
			By the minimality of $\wdh{\alpha}(I)$, we get that $\wdh{\alpha}(I) = \min\left\{ \frac{\alpha(I^{(a_1)})}{a_1}, \frac{\alpha(I^{(a_2)})}{a_2}\right\}$. Since both $r_1$ and $r_2$ are strictly smaller than $r$, this contradicts the minimality of $r$. \end{proof}
		
		We now provide more explicit formulas for \Cmatroid ideals. 
		
		\begin{Corollary}\label{Waldschmidt-SqFree} 
			Let $I$ be a \Cmatroid ideal with  $\h\,I=c$. 
			Then 
			$$\wdh{\alpha}(I) = \min_{1 \leq \ell \leq c}\left\{\frac{\alpha(\sfp_\ell(I))}{\ell}\right\}.$$
			Write $I=I_\Delta=J(\Delta^*)$ for some matroid $\Delta$. Then, additionally,
			\begin{enumerate}
				\item letting $c_h$ denote the smallest size of a circuit in the elongation of $\Delta$ to rank $n-c+h-1$,  then 
				$$
				\alpha(I^{(\ell)})= \min\left\{\sum_{h=1}^c a_hc_h\,\mid\,\sum_{h=1}^cha_h = \ell\right\},
				$$
				and
				$$\wdh{\alpha}(I) = \min\left\{ \frac{c_h}{h}\,\mid\,h=1,\ldots,c\right\};$$
				
				\item letting $f_h$ denotes the largest size of a flat of $\Delta^*$ of corank $h$, then
				$$
				\alpha(I^{(\ell)})= \min\left\{\sum_{h=0}^{c-1} a_h(n-f_h)\,\mid\,\sum_{h=0}^{c-1}(c-h)a_h = \ell\right\},
				$$
				and
				$$\wdh{\alpha}(I) = \min\left\{ \frac{n - f_h}{c - h}\,\mid\,h=0,1,\ldots,c-1\right\}.$$
				
			\end{enumerate}  
			
		\end{Corollary}
		
		We actually refine this result in Proposition \ref{Waldschmidt-UpTo-StarThreshold}, where we show one only needs to take the minimum over all $\ell$ between 1 and the \unif threshold of $I$. (see Definition \ref{Star-Threshold-Def}.) 

		\begin{proof}  The first equality and the first equality of (2)  follow by applying first Proposition \ref{SymReesAlgDescription} and then Proposition \ref{Wadlschmidt-Fin-Gen-Sym-Alg}. The second equality in (2) follows by the above and Proposition \ref{SquareFree-Correspond-To-Flats}.
			The equalities of (1) follow by the ones of (2), by applying the duality of elongation and truncation, and using the well--known facts that the flats of $\Delta$ of rank $r(\Delta) -h $ are the hyperplanes of the $h$-th truncation of $\Delta$, and a subset $F\subseteq [n]$ is a hyperplane of $\Delta$ if and only if its complement $[n]-F$ is a circuit of the dual matroid. (see e.g. \cite[Prop.~2.1.6]{Oxley}.) 
		\end{proof}
		
		We will see in the next section that for paving and sparse paving matroids one can use the above general formula to obtain simpler, more explicit formulas, see for instance  Proposition \ref{Paving-Equiv}, Corollary \ref{Waldschmidt-Paving}, and Proposition \ref{DegreesOfSparsePavingMatroids}. 
		
		The above results can also be employed to produce bounds on the resurgence of \Cmatroid ideals. 
		First, we recall the definition of {\em resurgence} of an ideal $I$:
		$$
		\rho(I)=\sup\{\ell/r\,\mid\,I^{(\ell)}\not\subseteq I^r\}.
		$$
		This invariant was introduced by Bocci and Harbourne in \cite{BH}. The resurgence and related invariants (e.g. the asymptotic resurgence, or the ic-resurgence) have been object of recent research, see for instance \cite{GHMres}, \cite{DD}, \cite{JKM}, \cite{GHMexpres}, \cite{Vi}, \cite{KNT} and references therein. 
		
		When $I$ is an ideal whose associated primes have the same height $c$, combining a theorem by Bocci and Harbourne \cite[Thm~1.2.1(a)]{BH} with a theorem by Ein--Lazarsfeld--Smith \cite[Thm~A]{ELS} and Hochster--Huneke \cite[Thm~1.1(a)]{HH} one obtains the inequalities
		$$
		\frac{\alpha(I)}{\widehat{\alpha}(I)}\leq \rho(I) \leq c.
		$$
		When $I$ is the defining ideal of a 0-dimensional scheme, e.g. a set of points in a projective space, then, according to \cite[Thm~1.2.1]{BH}, one has a stronger upper bound: 
		$$
		\frac{\alpha(I)}{\widehat{\alpha}(I)}\leq \rho(I) \leq \frac{{\rm reg}(I)}{\widehat{\alpha}(I)}.
		$$
		We then obtain the following formula for the resurgence of specializations of \Cmatroid  ideals (see Definition \ref{Def-Matroidalconfig}). For simplicity, we state it only for linear specializations. 
		
		\begin{Corollary}
			Let $I_*$ be a linear specialization of a \Cmatroid ideal. Let $\Delta$ be the matroid such that $I=I_\Delta$ and let $c_h$ denotes the smallest size of a circuit in the elongation of $\Delta$ to rank $n-c+h-1$. Then
			$$
			\max\left\{\frac{h\cdot c_1}{c_h}\,\mid\,h=1,\ldots,c\right\}\leq \rho(I_*^{(\ell)})\leq  c.
			$$
			If, furthermore, $I_*$ is the ideal of a matroidal configuration of points in $\mathbb P^c$, then
			$$
			\max\left\{\frac{h\cdot c_1}{c_h}\,\mid\,h=1,\ldots,c \right\} \leq \rho(I_*) \leq \max\left\{\frac{h\cdot (n-c+1)}{c_h}\,\mid\,h=1,\ldots,c \right\}
			$$

		\end{Corollary}
		
		We state this corollary by taking $I$ to be $I_\Delta$ because the result for $I=J(\Delta)$ is slightly less elegant. However, the interested reader can use the above discussion and Corollary \ref{Waldschmidt-SqFree} (2) to deduce it.

		\section{Formulas for Paving and Sparse Paving Matroids}\label{Section-Paving}
		
		In this section we obtain simplified formulas in terms of the {\em \unif threshold}, which we introduce as a measure of how far a matroid is from a uniform matroid, and may be of independent interest. In particular, we obtain easy--to--use formulas for paving and sparse paving matroids. 
		
		In general for any matroid $\Delta$ the size of of its circuit is at most $r(\Delta) + 1$. 
		Roughly speaking, a paving matroid is a matroid with only ``large" circuits. 
		
		\begin{Definition}\label{Def-Paving}
			A matroid $\Delta$ is {\em paving} if all its circuits have size greater than or equal to $r(\Delta)$. A {\em sparse paving} matroid is a paving matroid $\Delta$ whose dual $\Delta\dual$ is paving too.
		\end{Definition}
		
		By definition, paving matroids $\Delta$ of rank $r(\Delta)> 1$ are loopless and, similarly, sparse paving matroids of rank $r(\Delta) > 1$ are loopless and coloopless.
		
		\begin{Remark}\label{Rmk-Conj}
			By Definition \ref{Def-Paving} and Proposition \ref{Prop-2-Loc}, every paving matroid of rank $>2$ has a 2--locally connected circuit graph $G(\Delta)$.
			
			Since, conjecturally, almost all matroids are paving, see e.g. \cite[Conj~1.6]{MNWW} and \cite[Section~6]{LOSW}, then, conjecturally, the Stanley--Reisner and cover ideals of almost every matroid $\Delta$ of rank $>2$ have maximal symbolic defects.
		\end{Remark}

		The uniform matroid on $[n]$ of rank $c$ is a sparse paving matroid for any $1\leq c\leq n$. Its cover ideal is known as a {\em monomial star configuration ideal} of height $c$ in $n$ variables. For simplicity, for the rest of this paper, by {\em star configuration} we mean a monomial star configuration ideal.
		
		\begin{Proposition}\label{Paving-Equiv} 
			Let $\Delta$ be a matroid with cover ideal $J = J(\Delta)$, and let $c = \h\, J =r(\Delta)$ with $c >1$. The following are equivalent:
			\begin{enumerate}
				\item $\Delta$ is a paving matroid,
				\item $\sfp_\ell(J)$ is a star configuration of height $c-\ell+1$ for $2 \leq \ell \leq c$,
				\item $\alpha(\sfp_\ell(J))=n-c+\ell$, for all $2\leq \ell \leq c$,
				\item $\sfp_2(J)$ is a star configuration of height $c-1$,
				\item $\alpha(\sfp_2(J))=n-c+2$.
			\end{enumerate}
			
			\begin{proof} $(1) \Longrightarrow (2):$ By definition, the circuits of $\Delta$ are the minimal dependent set of $\Delta$. Hence, as $\Delta$ is paving, every subset of size at most $r(\Delta)-1=c-1$ is independent. 
				
				$(2) \Longrightarrow (4)$ is clear. $(4) \Longrightarrow (1)$ follows from Proposition \ref{SquareFreePartIsSkeleton}. $\sfp_2(J)$ is the cover ideal of the first truncation of $\Delta$. Hence all subset of size $c-1$ are independent in $\Delta$, which shows that the circuits of $\Delta$ must have size greater than or equal to $c = r(\Delta)$.
				
				$(2)\Lra (3)$ and $(4)\Lra(5)$ are well--known (see e.g. \cite[Cor~2.4 or Cor~2.5]{PS} for a stronger statement). For the converse, we show that if a matroid $\Gamma$ of rank $r(\Gamma)$ on $[n]$ is not uniform of rank $c$, then $\alpha(J(\Gamma))\leq n-c$. To see this, let $F=\{i_1,\ldots,i_c\}\notin \F(\Gamma)$, then $x_{[n]-F}\in J(\Gamma)$ proving $\alpha(J(\Gamma))\leq n-c$.
			\end{proof}
		\end{Proposition}

		So, algebraically, the squarefree part of $\symp{J(\Delta)}$ can be effectively used to  measure how far away a pure simplicial complex $\Delta$ is from being a uniform matroid.
		We now formalize this intuitive notion of ``being close to a uniform matroid".
		
		\begin{Definition}\label{Star-Threshold-Def} Let $\Delta$ be a pure simplicial complex. We define the {\em \unif--threshold} $\ut(\Delta)$ of $\Delta$ to be the minimum integer $u$ for which $\Delta_{\dim(\Delta)-u}$, i.e. the codimension $(u-1)$-skeleton of $\Delta$, is a uniform matroid. If no such $u$ exists, then we set the \unif threshold to be $\infty$.
		\end{Definition}
		
		Note that $\ut(\Delta)=\infty$ if and only if there is a vertex in $[n]$ which does not belong to $\Delta$ (a loop, if $\Delta$ is a matroid). If there is no such vertex and $\dim(\Delta)=c-1$, then $\ut(\Delta)\leq c$ because at least for $u = c$, the $0$-skeleton of $\Delta$, i.e. the set of vertices of $\Delta$, is the uniform matroid of rank $1$. 
		
		\begin{Remark} 
			Let $J$ be the cover ideal of a pure simplicial complex $\Delta$, with $c := \h\,J = \dim(\Delta) + 1$. Then, $\Delta$ is a uniform matroid $\Llra$ $\ut(\Delta)=1$, by Proposition \ref{Paving-Equiv}, $\Delta$ is paving $\Llra$ $\ut(\Delta)\leq 2$, and $\Delta$ is simple $\Llra$ $\ut(\Delta)\leq c-1$.
			
			Also, by Proposition \ref{SquareFreePartIsSkeleton} the cover ideal of the $(c - \ell)$-skeleton of $\Delta$ is $\sfp_{\ell}(J)$, thus 
			$$\ut(\Delta)=\min\left\{u\in \ZZ_+\,\mid\, \sfp_{\ell}(J) \text{ is a star configuration for every }\ell\geq u\right\}.$$
			
		\end{Remark}
		
		We now establish a connection between the \unif threshold and the girth of any matroid. 
		As $\ut(\Delta)=\infty$ if $\Delta$ is a matroid containing a loop, we assume $\Delta$ is loopless. In general, it is easy from an algebraic perspective to deal with matroids with loops or coloops, see Proposition \ref{Basic-Matroid-Properties}(6). 
		
		\begin{Proposition}\label{Star-Threshold-Girth} Let $\Delta$ be a loopless matroid of rank $c$ on $[n]$ with $c < n$. Let $g$ be the girth of $\Delta$, i.e. the minimal size of a circuit of $\Delta$. Then $\ut(\Delta)=c+2-g$.
			
		\end{Proposition}
		
		\begin{proof} 
			Since $\Delta$ is loopless, then $g\geq 2$. To establish the result we show that the $(g-2)$-skeleton of $\Delta$ is a uniform matroid, and that the $(g-1)$-skeleton is not a uniform matroid.
			
			By definition of girth, any subset $A\sub [n]$ of size $g-1$, hence of $\dim(A)=g-2$, is independent. Thus, $\Delta_{c+2-g}$ is a uniform matroid. 
			Now, the circuit $C$ of $\Delta$ of smallest size $g$, is a subset of $[n]$ of $\dim(C)=g-1$ and $C\notin \Delta$, so $\Delta_{c+1-g}$ is not a uniform matroid.
		\end{proof}

		In the next result we show how the \unif threshold allows us to refine the formula for the Waldschmidt constant in Corollary \ref{Waldschmidt-SqFree} by reducing the number of terms for which one takes the minimum. Since $\widehat{\alpha}(J(\Delta))=n/c$ if $\Delta$ is the uniform matroid of rank $c$ on $[n]$, in the next result we may further assume $\Delta$ is not the uniform matroid. 
		
		\begin{Proposition}\label{Waldschmidt-UpTo-StarThreshold}  Let $J = J(\Delta)$ be the cover ideal of a loopless matroid $\Delta$ of rank $c$ over $[n]$. If $\Delta$ is not the uniform matroid, then with $u = \ut(\Delta)$
			$$\wdh{\alpha}(J) = \min \left\{ \alpha(J), \frac{\alpha(\sfp_2(J))}{2}, \dots, \frac{\alpha(\sfp_{u-1}(J))}{u-1}, \frac{n}{c} \right\}.$$
		\end{Proposition}
		
		\begin{proof} By Corollary \ref{Waldschmidt-SqFree} $\wdh{\alpha}(J) = \min_{1 \leq \ell \leq c}\left\{\frac{\alpha(\sfp_\ell(J))}{\ell}\right\}$. 
			Whenever $\sfp_\ell(J)$ is a star configuration of height $c-\ell+1$, then $\alpha(\sfp_\ell(J)) = (n-c+\ell)/\ell$. Observe that $(n-c+\ell)/\ell$ is a non--increasing sequence, therefore
			$$\wdh{\alpha}(J) = \min \left\{ \alpha(J), \frac{\alpha(\sfp_2(J))}{2}, \dots, \frac{\alpha(\sfp_{u-1}(J))}{u-1}, \frac{n - c +u}{u} , \dots, \frac{n-1}{c-1}, \frac{n}{c} \right\}$$$$ = \min \left\{ \alpha(J), \frac{\alpha(\sfp_2(J))}{2}, \dots, \frac{\alpha(\sfp_{u-1}(J))}{u-1}, \frac{n}{c} \right\}.$$
		\end{proof}
		
		As an application, we can quickly compute $\wdh{\alpha}(I)$ when $\Delta$ is a paving matroid. 
		\begin{Corollary}\label{Waldschmidt-Paving}
			Let $J = J(\Delta)$ be the cover ideal of a non--uniform paving matroid $\Delta$ of rank $c$ over $[n]$. Then $$\wdh{\alpha}(J) = \min \left\{ \alpha(J), \frac{n}{c}\right\}.$$
		\end{Corollary}
		
		\begin{proof} Since $\Delta$ is not uniform then $\ut(\Delta) \geq 2$. By Proposition \ref{Paving-Equiv}, $\ut(\Delta)\leq 2$, thus $\ut(\Delta)=2$. The result now follows from Proposition \ref{Waldschmidt-UpTo-StarThreshold}.
		\end{proof}
		
		\begin{Example} Consider the following rank $3$ matroid on $[6]$, $$\Delta = \{\{1, 2, 5\}, \{1, 2, 6\}, \{1, 3, 5\}, \{1, 3, 6\}, \{1, 4, 5\}, \{1, 4, 6\}, \{1, 5, 6\}, \{2, 3, 5\},$$ $$ \{2, 3, 6\}, \{2, 4, 5\},\{2, 4, 6\}, \{3, 4, 5\}, \{3, 4, 6\}, \{3, 5, 6\}, \{4, 5, 6\}\}.$$ One can check that $\Delta$ is paving, but not sparse paving -- because $\{5,6\}$ is a cocircuit. Its cover ideal is $J = (x_5x_6, x_1x_3x_4, x_2x_3x_4x_6, x_1x_2x_4x_6, x_1x_2x_3x_6, x_2x_3x_4x_5, x_1x_2x_3x_5)$. 
			By Proposition \ref{Paving-Equiv}, or by Proposition \ref{Star-Threshold-Girth} (because the girth of $\Delta$ is 3), we see that $\ut(\Delta)=2$. 
			Applying Corollary \ref{Waldschmidt-Paving} we get $\wdh{\alpha}(J) = 2$.
			
		\end{Example}
		
		When $\Delta$ is a sparse paving matroid, we can completely describe $\alpha(J^{(\ell)})$ for all $\ell$. If $\Delta$ is a uniform matroid, then $\alpha(\symp{J})$ is already known. (see e.g. \cite[Thm~7.7]{M} for a stronger statement.) 
		So we avoid this case in the next result. 
		
		\begin{Proposition}\label{DegreesOfSparsePavingMatroids} Let $J$ be the cover ideal of a non-uniform sparse paving matroid  $\Delta$ of rank $c$ on $[n]$. If $\Delta$ has a coloop then $\alpha(\symp{J}) = \ell$ for all $\ell$. If $\Delta$ is coloopless, then 
			$$\left\{
			\begin{array}{lll}
				\alpha(J) & = n - c,&\\
				\alpha(\symp{J}) &= n - c + \ell,& \text{ for } 2 \leq \ell \leq c \,.
			\end{array}\right.$$
			For $ \ell > c$, write $\ell = qc + b$ for $1 \leq b \leq c$, then $$\alpha(\symp{J}) = q\alpha(J^{(c)}) + \alpha(J^{(b)}) = \begin{cases} (q+1)n - c + b & \textrm{ if } b \neq 1 \\
				(q+1)n - c & \textrm{ if } b = 1 
			\end{cases}.$$
			
		\end{Proposition}
		
		\begin{proof} First, assume $\Delta$ has a coloop, then $J$ contains a variable $x$. Since $x^\ell\in \symp{J}$, then $\alpha(\symp{J})\leq \ell$, and since the other inequality is trivial, then  $\alpha(\symp{J}) = \ell$ for all $\ell$. 
			
			Next, assume $\Delta$ is coloopless. Observe that since $\Delta$ is not uniform, then $n-c\geq 2$. 
			
			For $2 \leq  \ell \leq c$ we prove by induction that $\alpha(\symp{J}) = n-c+\ell$. When $\ell=2$, by Proposition \ref{Paving-Equiv}, $\sfp_2(J) = n - c + 2$. From the Structure Theorem \ref{MatroidSymPowerThm} it follows that $\alpha(J^{(2)}) = \min \{ 2\alpha(J), \alpha(\sfp_2(J)) \} = \min \{ 2n-2c, n - c + 2 \}$, which is $n - c + 2$, since $n-c\geq 2$.
			
			If $\ell > 2$, observe that $\alpha(\symp{J}) > \alpha(J^{(\ell-1)}) = n - c + \ell - 1$. On the other hand $\alpha(\symp{J}) \leq \alpha(\sfp_\ell(J)) = n - c + 
			\ell$. Hence $\alpha(\symp{J}) = \alpha(\sfp_\ell(J))=n-c+\ell$. This concludes the case where $2 \leq \ell \leq c$.
			
			If $\ell > c$, by the Structure Theorem \ref{MatroidSymPowerThm} we have $$\alpha(\symp{J}) = \min \left\{ \sum_{1\leq c_i<c} \alpha(J^{(c_i)}) \, | \, \sum_{1\leq c_i<c} c_i = \ell\right\}.$$ The desired formula for $\alpha(\symp{J})$ will follow from the above description and iterated applications of the following inequality, which we will show. Let $a$ and $b$ be integers such that $1 \leq a \leq b < c$ and $a + b \geq c$, then $$\alpha(J^{(a)}) + \alpha(J^{(b)}) \geq \alpha(J^{(c)}) + \alpha(J^{(a + b - c)}).$$
			
			First, assume $a = 1$. In this case the inequality, $a + b \geq c$ and $b < c$, forces $b = c - 1$. This shows that $a+b-c = 0$, hence the displayed inequality above holds from subadditivity of $\alpha$. Next, assume $a > 1$. Then as $2 \leq a \leq b < c$, from the first result shown above, we know $$\alpha(J^{(a)}) + \alpha(J^{(b)}) = 2n - 2c + a + b.$$
			Now $1 \leq a + b - c < c$, again from the first result, we have $\alpha(J^{(a + b - c)})$ is at most $ n - 2c + a + b$. So then $\alpha(J^{(c)}) + \alpha(J^{(a + b - c)}) \leq 2n - 2c + a + b$ as desired. 
		\end{proof}
		
		\begin{Corollary}\label{Waldschmidt-SparsePaving} 
			Let $J$ be the cover ideal of a non-uniform sparse paving matroid  $\Delta$ of rank $c$ on $[n]$. If $\Delta$ has a coloop then $\wdh{\alpha}(J) = 1$. If $\Delta$ is coloopless, then $\wdh{\alpha}(J) = n/c$. 
		\end{Corollary}
		
		\begin{proof} If $\Delta$ has a coloop, the equality follows from Proposition \ref{DegreesOfSparsePavingMatroids}. If $\Delta$ is coloopless then $n-c\geq 2$ and, by Proposition \ref{DegreesOfSparsePavingMatroids} and Corollary \ref{Waldschmidt-Paving}, $$\wdh{\alpha}(J) = \min \left\{ \alpha(J), \frac{n}{c}\right\} = \min \left\{ n-c, \frac{n}{c}\right\}.$$
			The statement follows because $n/c \leq n-c$ precisely when $n-c\geq 2$. 
		\end{proof}
		
		In particular, we can find much narrower bounds for the resurgence of matroidal configurations of points obtained by specializing 
		the cover ideal of a sparse paving matroid. 
		\begin{Corollary}\label{ResurgenceSparsePaving}
			Let $\Delta$ be a sparse paving matroid of rank $c$ on $[n]$ with no loops or coloops. Assume $\Delta$ is not a uniform matroid. Let $-_*$ denote the specialization to a matroidal configuration of points. Then
			$$
			\frac{c(n-c)}{n} \leq \rho(J(\Delta)_*)\leq \frac{c(n-c+1)}{n}.
			$$
		\end{Corollary}
		
		Corollary \ref{ResurgenceSparsePaving} generalizes \cite[Prop.~3.8]{BFGM}, because the matroids studied in \cite{BFGM} are sparse paving. This latter fact is well--know but, for the sake of completeness, we give a short proof below. 
		
		\begin{Definition} 
			A {\em Steiner system} with parameters $S(n,d,t)$, is a pair $(V, \B)$, where 
			\begin{itemize}
				\item $V$ is a set of $n$ elements, and
				\item $\B$ is a collection of subsets of $V$ with $t$ elements,  called the {\em blocks} of the system, where every subset of $V$ of size $d$ is contained in exactly one block.
			\end{itemize} 
			
		\end{Definition}
		
		In \cite{BFGM}, the authors construct a matroid $\Delta$ from a Steiner system $(V,\B)$ with parameters $S(n,d,t)$ with $d<t$ in the following way. $V$ is the groundset of $\Delta$, and the basis of $\Delta$ are all the $t$-subsets of $V$ that are not blocks. 
		
		\begin{Remark} In the matroid literature, there are different ways to construct a matroid from a Steiner system. For instance, instead of the above, one could take the set $\B$ to be the set of hyperplanes of the matroid, see \cite{Welsh} and \cite{Hofstad2018}. In general, this construction and the one in \cite{BFGM}, lead to different matroids. They coincide if and only if $d = t -1$.
		\end{Remark}
		
		\begin{Proposition}
			Given a Steiner system  with parameters $S(n,d,t)$ with $d<t$, the matroid $\Delta$ associated to it following \cite{BFGM} is a sparse paving matroid.
		\end{Proposition}
		
		\begin{proof}
			We may assume $d < t < n$ and the groundset is $V=[n]$. We first show that $\Delta$ is paving using Proposition \ref{Paving-Equiv} (4). Let $J=J(\Delta)$, and observe that  $t = \h\,J = r(\Delta)$. By Corollary \ref{SqFree-Skeleton-Elongation} we know that $\sfp_2(J)$ is the cover ideal of $\Delta_{t-1}$, so we need to show that $\Delta_{t-1}$ is the uniform matroid of rank $t-1$. Let $A\subseteq [n]$ with $|A|=t-1$, then there are at least two distinct $t$--subsets $B_1$ and $B_2$ of $[n]$ containing $A$. Observe that $B_1$ and $B_2$ cannot both be blocks because otherwise $B_1 \cap B_2 \supseteq A$, and since $|A|=t-1\geq d$, then $B_1$ and $B_2$ share a subset of size $d$ -- a contradiction. Hence either $B_1$ or $B_2$ is a basis, and since $A$ is contained in it, then $A\in \Delta$.
			
			Next we show $\Delta\dual$ is paving. By Corollary \ref{SqFree-Skeleton-Elongation} and Proposition \ref{Paving-Equiv} (4) it suffices to show that $\Delta^1$, i.e. the $1$-elongation of $\Delta$, is the uniform matroid. 
			Let $A$ be any subset of size $t+1$, we show that $A$ is a basis of $\Delta^1$ by showing that $A$ contains a basis of $\Delta$. 
			Suppose not, then all $t$--subsets of $A$ are blocks,  
			then any $d$--subset $D$ of $A$ is contained in any $t$--subset of $A$ containing $D$. Since all $t$--subsets of $A$ are blocks and $|A| = t +1 \geq d + 2$, there is at least two blocks containing $D$, a contradiction.
		\end{proof}
		
		\begin{Example}\label{Ex-Fano-Plane} The Fano Plane $F_7$ is a Steiner system with parameters $S(7,2,3)$. Up to isomorphism it is the only Steiner system with those parameters. Its blocks can be taken to be 
			$$\B = \{ \{1,2,3\}, \{1,4,7 \}, \{1,5,6 \}, \{2,4,6\}, \{2,5,7\}, \{3,4,5\}, \{3,6,7\}\}.$$ 
			The bases of the corresponding matroid $\Delta$ are the $3$-subsets of $[6]$ which are not blocks. It has cover ideal $$J(
			\Delta) = (x_4x_5x_6x_7,\; x_2x_3x_5x_6,\; x_2x_3x_4x_7,\; x_1x_3x_5x_7,\;  x_1x_3x_4x_6,\; x_1x_2x_6x_7,\; x_1x_2x_4x_5).$$ $\Delta$ is a sparse paving matroid so, by Corollary \ref{Waldschmidt-SparsePaving}, $\wdh{\alpha}(J(\Delta)) = 7/3$.
		\end{Example}
		
		\section{Applications 3: A fast Macaulay2 algorithm for computing $\symp{I}$ when $\ell\gg0$}\label{Section-Alg} 
	
	\texttt{Macaulay2} has a function, \texttt{symbolicPower}, in the \texttt{SymbolicPowers} package, which computes $\symp{I}$ for any homogeneous ideal $I\subseteq R$. There is an optional argument \texttt{CIPrimes => true} which can be applied for squarefree monomial ideals $I$.
	The functions require small amount of time when $\ell$ and $\dim(R)$ are small. 
	
	For \Cmatroidal ideals $I$, we employ the Structure Theorem \ref{MatroidSymPowerThm} to provide an alternative way to compute $\symp{I}$. 
	
	Starting from relatively small $\ell$, our algorithm is significantly faster than \texttt{symbolicPower}. In fact, it allows very fast computations of large symbolic powers. For instance, the table below shows the computation timings for the cover ideal $J$ of the Fano plane in Example \ref{Ex-Fano-Plane}. We denote our algorithm as \texttt{symPowMatroid}. All times are in seconds. \begin{table}[h!] 
		$$\begin{array}{l|c|c }
			\symp{J}& \texttt{symbolicPower} & \texttt{symPowMatroid} \\ \hline
			\ell=  10 &1.34375 & 0.15625    \\
			\ell=  11& 2.42187 & 0.17187\\
			\ell= 12& 4.32812 & 0.20312\\
			\ell=  13& 7.60937&  0.23437\\
			\ell=   14& 12.4375& 0.28125\\
			\ell=   15& 20.0625&0.29687 \\
			\ell=    16& 31.9687& 0.35937\\
			\ell=    17& 51.3906& 0.40625\\
			\ell=    18& 74.2812& 0.45312\\
			\ell=    19& 110.984& 0.48437 \\
			\ell=    20& 167.406&  0.54687 \\
		\end{array}
		$$\end{table}

	From the table, the time that \texttt{symbolicPower} takes to compute $\symp{J}$ appears to be exponential in $\ell$, whereas for \texttt{symPowMatroid} the timings grow {\em linearly} in $\ell$. The time difference is pronounced for very large $\ell$. Our algorithm computes $J^{(100)}$ in our computer in about $13$ seconds, it can be surmised from the growth in the table that \texttt{symbolicPower} will take a few hours.
	
	Here we will present the algorithm in pseudocode. A direct implementation in Macaulay2 of the algorithm can be found on the second author's website. 
	
	To simplify the presentation of the algorithm, we illustrate it for $J$ the cover ideal of a matroid of rank $c$. With small variations one obtains the analogous algorithm for the Stanley--Reisner ideal of a matroid. Alternatively, one can use the algorithm for the cover ideal of the dual matroid, which would be the Stanely--Reisner ideal of the matroid. 
	
	The Structure Theorem \ref{MatroidSymPowerThm} states that all minimal generators of $\symp{J}$ are either squarefree, or are obtained by multiplying squarefree minimal generators in lower symbolic powers. The first step in the algorithm is then to compute the squarefree part of each $\symp{J}$, denoted as $\sfp_\ell(J)$, for $1 \leq \ell \leq c$. One way to obtain it is to first use \texttt{symbolicPower}$(J,\ell)$ for $1 \leq \ell \leq c$, and then use \texttt{squarefreeGens} to take the squarefree minimal generators. However, by Proposition \ref{SquareFreePartIsSkeleton}, we can instead compute the cover ideal of the skeletons of the simplicial complex of $J$. For brevity, we omit this implementation and assume we have done so with the function \texttt{skeleton}$(J,\ell)$.
	
	Now, from the Structure Theorem \ref{MatroidSymPowerThm}, to compute $\symp{J}$, we loop through all possible symbolic types -- recall that they are sequences $(c_1,c_2,...,c_s)$, where $\sum c_i = \ell$ and $c_1 \geq c_2 \geq ... \geq c_s\geq 1$. We may view these sequences as partitions of $\ell$ with parts of size at most $c$. The set of all such partitions can be computed in \texttt{Macaulay2} using the function \texttt{partitions}. To keep track of repeating $c_i$'s we write the sequence $(c_1,c_2,...,c_s)$ as $((c_1, n_1), (c_2,n_2),...,(c_t,n_t))$, where each $n_i$ is the number of times $c_i$ appears. Now, for each partition, we will multiply along powers of minimal generators $m_i^{n_i} \in \sfp_{c_i}(J)$ where the supports of the $m_i$'s form descending chains. One way to algorithmically carry out all such possible multiplication, is by forming the order complex on $\cup_{i=1}^{t} \sfp_{c_i}(J)$ where the partial ordering is given by divisibility. In \texttt{Macaulay2} this can be carried out with the package \texttt{Poset}. For brevity we omit the algorithm for forming this complex, and we will wrap it in the method \texttt{orderComplex}. After forming the complex, we obtain all generators with symbolic type $(c_1,...,c_s)$ by multiplying along all facets of the order complex, and raising each monomial part to the appropriate power.
	
	\begin{algorithm}[H]
		\caption{symPowMatroid($J$,$\ell$)}
		\begin{algorithmic}
			\State \textbf{Input:} matroid cover ideal $J$, integer $\ell$
			\State \textbf{Output:} $L = J^{(\ell)}$
			\State $c \gets \min\{\h\, J, \ell\}$
			
			\State $\mathcal{P} \gets$ partitions$(\ell, c)$
			\State $L \gets$ ideal $0$
			\State \textrm{sqfreeParts} $\gets$ empty list
			\For{$h$ from $1$ to $c$}
			\State \textrm{sqfreeParts} $\gets$ append$($\textrm{sqfreeParts}, skeleton$(J,h))$
			\EndFor
			\For{\textbf{each} partition $P$ in $\mathcal{P}$}
			\State $\Gamma \gets$ orderComplex$(\textrm{sqfreeParts},P)$ 
			\For{\textbf{each} unique $i$ in $P$}
			\State $n_i \gets $ count $i$ in $P$
			\EndFor
			\State $M \gets 1$
			\For{\textbf{each} facet $F$ in $\mathcal{F}(\Gamma)$}
			\For{\textbf{each} monomial $m$ in $F$} 
			\State $M \gets M \times m^{n_i}$
			\EndFor
			\State $L \gets L$ + ideal $M$
			\EndFor
			\EndFor \\
			\Return $L$
		\end{algorithmic}
	\end{algorithm}

\bibliographystyle{amsalpha}
	\bibliography{MatroidBib}
	
\end{document}